\documentclass[a4paper,11pt]{amsart}

\usepackage[textsize=footnotesize,color=blue!40, bordercolor=white]{todonotes}

 \usepackage
[a4paper, margin=1.2in]{geometry}

\usepackage{hyperref} 

\usepackage{xcolor}
\definecolor{blue}{rgb}{0,0,1}
\hypersetup{
	unicode=true,
	colorlinks=true,
	citecolor=blue,
	linkcolor=blue,
	anchorcolor=blue
}

\usepackage[latin1]{inputenc}
\usepackage[T1]{fontenc}
\usepackage[english]{babel}

\usepackage{mathrsfs}
\usepackage{amscd}
\usepackage{amsfonts}
\usepackage{amsmath}
\usepackage{amssymb}
\usepackage{amstext}
\usepackage{amsthm}
\usepackage{amsbsy}

\usepackage{xspace}
\usepackage[all]{xy}
\usepackage{graphicx}
\usepackage{url}
\usepackage{latexsym}

\makeatletter
\newcommand*{\rom}[1]{\expandafter\@slowromancap\romannumeral {\sharp}1@}
\makeatother

\theoremstyle{definition}

\newcommand{\Ord}{\mathrm{Ord}} 
\newcommand{\cof}{\mathrm{cof}} 
\newcommand{\HOD}{\mathrm{HOD}} 
\newcommand{\ZFC}{\mathsf{ZFC}}

\newcommand{\CC}{\mathcal{C}} 
\newcommand{\RR}{\mathcal{R}}

\newtheorem{fact}{fact}[section]

\newtheorem{thm}[fact]{Theorem}
\newtheorem{lemma}[fact]{Lemma}

\newtheorem{definition}[fact]{Definition}
\newtheorem{defini}[fact]{Definition}
\newtheorem{remark}[fact]{Remark}

\newtheorem{conjecture}[fact]{Conjecture}
\newtheorem{question}[fact]{Question}

\newtheorem{claim}[fact]{Claim}
\newtheorem*{claim*}{Claim}
\newtheorem*{subclaim*}{Subclaim}

\usepackage{amsmath,amssymb,amscd,amsthm,amsfonts,amstext,amsbsy,mathrsfs,hyperref,upgreek,mathtools,stmaryrd,enumitem,bbm}

\newenvironment{enumerate-(a)}{\begin{enumerate}[label={\upshape (\alph*)}, leftmargin=2pc]}{\end{enumerate}}

\newenvironment{enumerate-(a)-r}{\begin{enumerate}[label={\upshape (\alph*)}, leftmargin=2pc,resume]}{\end{enumerate}}

\newenvironment{enumerate-(A)}{\begin{enumerate}[label={\upshape (\Alph*)}, leftmargin=2pc]}{\end{enumerate}}

\newenvironment{enumerate-(A)-r}{\begin{enumerate}[label={\upshape (\Alph*)}, leftmargin=2pc,resume]}{\end{enumerate}}

\newenvironment{enumerate-(i)}{\begin{enumerate}[label={\upshape (\roman*)}, leftmargin=2pc]}{\end{enumerate}}

\newenvironment{enumerate-(i)-r}{\begin{enumerate}[label={\upshape (\roman*)}, leftmargin=2pc,resume]}{\end{enumerate}}

\newenvironment{enumerate-(I)}{\begin{enumerate}[label={\upshape (\Roman*)}, leftmargin=2pc]}{\end{enumerate}}

\newenvironment{enumerate-(I)-r}{\begin{enumerate}[label={\upshape (\Roman*)}, leftmargin=2pc,resume]}{\end{enumerate}}

\newenvironment{enumerate-(1)}{\begin{enumerate}[label={\upshape (\arabic*)}, leftmargin=2pc]}{\end{enumerate}}

\newenvironment{enumerate-(1)-r}{\begin{enumerate}[label={\upshape (\arabic*)}, leftmargin=2pc,resume]}{\end{enumerate}}

\title[Recognizable sets and Woodin cardinals]{Recognizable sets and Woodin cardinals: Computation beyond the constructible universe} 

\author{Merlin Carl, Philipp Schlicht and Philip Welch}

\begin{document}

\maketitle  

\begin{abstract} 
We call a subset of an ordinal $\lambda$ \emph{recognizable} 
if it is the unique subset $x$ of $\lambda$ for which some Turing machine with ordinal time and tape, 
which halts for all subsets of $\lambda$ as input, 
halts with the final state $0$. 
Equivalently, such a set is the unique subset $x$ which satisfies a given $\Sigma_1$ formula in $L[x]$. 
We prove several results about sets of ordinals recognizable from 
ordinal parameters by ordinal time Turing machines.
Notably we show the following results from large cardinals. 
\begin{itemize} 
\item 
Computable sets are elements of $L$, while 
recognizable objects 
appear up to the level of Woodin cardinals. 
\item 
A subset of a countable ordinal $\lambda$ 
is in the \emph{recognizable closure} for subsets of $\lambda$ 
if and only if it is an element of $M^{\infty}$, 
\end{itemize} 
where $M^{\infty}$ denotes the inner model obtained by iterating the least measure of $M_1$ through the ordinals, and where the recognizable closure 
for subsets of $\lambda$ 
is defined by closing under relative recognizability for subsets of $\lambda$.
\end{abstract}

\setcounter{tocdepth}{2} 
\tableofcontents

\section{Introduction}

Infinitary machine models of computation provide an attractive approach to generalized recursion theory. 
The first such model, Infinite Time Turing Machines, was introduced 
by Hamkins and Lewis \cite{HaLe}.\footnote{ Work on this paper was partially done whilst the authors were visiting at the Isaac Newton Institute for Mathematical Sciences in the programme `Mathematical, Foundational and Computational Aspects of the Higher Infinite' (HIF), to which they are grateful. In addition the last author was a Simons Foundation Fellow during this period and gratefully acknowledges that Foundation's support.} 
A motivation for considering such machine models is that they capture the notion of an effective procedure in a more general sense than classical Turing machines, thus allowing effective mathematics of the uncountable \cite{EU}. 

Such models are usually obtained by extending the working time or the working space of a classical model of computation to the transfinite. The strongest such models considered so far, to our knowledge, 
are Ordinal Turing Machines (OTMs) and the equivalent Ordinal Register Machines (ORMs). These were defined and studied by Peter Koepke and others \cite{Ko,KoSy}. 
It is argued in \cite{Ca4} that OTM-computability adequately expresses the intuitive notion of an idealized computor working in transfinite time and space.

The sets of ordinals which are OTM-computable from ordinal parameters are simply the constructible sets of ordinals. 
This is rather restrictive and 
it was asked whether one could think of machines that have an extra function allowing them to go outside of $L$ into the core model $K$ (\cite{FW}). 
This idea suggests a strengthening of the underlying machine model. 

Here we follow a different approach and consider the notion of recognizability. 
This means that for some initial input, some program will stop with output $1$ if the input is the object in question, and stop with output $0$ otherwise. 
It is thus a form of implicit definability. 
To our knowledge this was independently first considered for OTM-computability by Dawson \cite{Da09}. He showed that the OTM-computable sets 
coincided with the recognisable sets, without allowing ordinal parameters. He showed that if this was relaxed to allow constructibly countable ordinal parameters, then still no non-constructible sets are recognisable. He further showed that $0^{{\sharp}}$, if it existed, was recognisable from some uncountable cardinal, and that adding Cohen reals over $L$ did not add recognisable sets. 
The fact that recognizability is strictly weaker than computability was called the \emph{lost melody phenomenon} \cite{HaLe} (see also \cite{Ca2}). 
For instance an OTM can recognize $0^{\sharp}$ from the parameter $\omega_{1}$  \cite{Ca2}, although $0^{\#}$ is not computable, 
and thus the question arises how far recognizability goes beyond $L$. 

Similar to computability, recognizability can be relativized. 
Intuitively, an object $x$ is recognizable relative to an object $y$ if this can be used to identify $x$. 
Imagine that we cannot recognize a radioactive stone, but we can 
identify a Geiger counter and use this to identify the stone. 
We can iterate relative recognizability and obtain the \emph{recognizable closure} $\CC$. 
Various foundational views consider mathematical objects as objects of an idealized cognitive agent. 
Such a view on set theory is entertained by Hao Wang (\cite{Wa}) and Philip Kitcher (\cite{Ki}) and is also present in various remarks by G\"odel. 
The recognizable closure can be interpreted as the range of objects that are recognized by an idealized agent. 
From this viewpoint, the notion of the recognizable closure may be relevant for the philosophy of mathematics. 

In this paper, we prove various results about recognizable sets of ordinals and about the recognizable closure. 
Recognizability generalizes the notion for reals of $\Delta^1_2$ in countable ordinals by allowing arbitrary ordinals. 
Indeed, if $\omega_{1}=\omega_{1}^{L}$ then the relativized relation between reals $x,y$ of $x$ being recognizable from $y$ in a countable ordinal coincides with $\Delta^{1}_{2}$ reducibility (see Corollary \ref{characterization of recognizable from alpha}), and the consequent degree structure is that of the $\Delta^1_{2}$-degrees.

While recognizability from fixed ordinal parameters is not absolute to generic extensions if $0^{\#}$ exists, the recognizable closure is more stable. 
We prove that if there is a measurable cardinal above a Woodin cardinal, then the recognizable closure for subsets of $\omega$ is exactly the set of reals in $M_1$, a fine structural model with a Woodin cardinal. 
This set is also known as $Q_3$ \cite{KeMaSo}. 
There are various other characterizations of $Q_3$, for instance as the maximal countable $\Pi^1_3$ set which is closed under $\Delta^1_3$-reducibility \cite[p. 202, section II]{KeMaSo} assuming projective determinacy. 
We prove this result for subsets of arbitrary countable ordinals. 
It is open whether our results can be extended to the recognizable closure for subsets of arbitrary cardinals. 

We further show that the generic version of the recognizable closure contains exactly the sets of ordinals in $M^{\infty}$, where $M^{\infty}$ is obtained by iterating the least measure in $M_1$ through the ordinals. 

The paper is structured as follows. Section 2 introduces Ordinal Turing Machines and basic results about recognizability. 
Section 3 contains results on recognizability with specific ordinal parameters and 
on the recognizable closure in generic extension. 
Section 4 then connects recognizability with inner models. 
The main result states that the recognizable closure for subsets of a countable ordinal $\alpha$ is the set of real numbers in $M^{\infty}$, if there is a measurable cardinal above a Woodin cardinal.

We would like to thank Gunter Fuchs, Daisuke Ikegami, Vladimir Kanovei, Ralf Schindler and John Steel for discussions related to the topic of this paper.

\section{Definitions and basic facts} 

Ordinal Turing Machines (OTMs) were introduced independently by Koepke and  Dawson, and appeared in \cite{Ko}, and the latter's unpublished thesis \cite{Da09}, as a further generalization of classical Turing machines to the transfinite, following the infinite time Turing machines (ITTMs) of Hamkins and Kidder. 
They provide an upper bound on the strength of a reasonable model of transfinite computation 
(see e.g. \cite{Ca4} for an argument in favor of this claim). 

We refer to \cite{Ko} for the definitions, but give a brief description of the model and its computations. 
An OTM has a tape whose cells are indexed with the ordinals and runs along an ordinal time axis. At each time, each cell contains either a $0$ or a $1$. An OTM-program is just an ordinary
Turing machine program. A computation state thus consists of the tape content, which is a function $t\colon \Ord\rightarrow 2$, the head position, an ordinal, and a program state. 
There are finitely many states and these are indexed by natural numbers. 

At a successor time $\alpha+1$, the computation state is determined from the state at time $\alpha$ in the same way as for an ordinary Turing machine with the supplement that, if the head is currently at a limit position
and is now supposed to move to the left, it is set to position $0$. At limit times, the content of each cell, the head position and the program state are obtained as the inferior limits of the sequence of earlier
cell contents, head positions and program states. A computation stops when it assumes a state for which no further state can be determined from the program. The computation can be given an ordinal parameter $\alpha$
by marking the cell at $\alpha$ with a $1$ before the computation starts.

A subset $X$ of an ordinal $\gamma$ is OTM-computable in the parameter $\alpha$ if there is an OTM-program $P$ which takes as input the value $1$ at $\alpha$ and $\beta<\gamma$ and the value $0$ otherwise, 
stops with $1$ on the first cell if $\beta\in X$, and stops with $0$ on the first cell if $\beta\notin X$. 
A set $X$ of ordinals is OTM-computable if it is OTM-computable from some ordinal. 
If the computation can be done without an ordinal parameter $\alpha$, then we call $X$ OTM-computable without ordinal parameters. 

For notation, suppose that $P$ is a program and $X$ is a set of ordinals. Then $P^{X}\downarrow=y$ means that $P$ stops with output $y$ with the oracle $X$. 
Moreover $P^{X}\uparrow$ means that the computation diverges. 
The \emph{Kronecker symbol} $\delta_{xy}$ is defined as $\delta_{xy}=1$ if $x=y$ and $\delta_{xy}=0$ otherwise. 
We denote a computation of $P$ with the oracle $x$ and the ordinal parameter $\alpha$ by $P^{x}(\alpha)$. 
Let $x\oplus y:=\{2n:n\in x\}\cup\{2n+1:n\in y\}$ denote the join of $x,y\subseteq\omega$. 

The main result of \cite{Ko}, 
independently obtained by Dawson \cite{Da09}, states that the OTM-computable sets of ordinals coincide with the constructible sets of ordinals. This result and its proof can be relativized in a straightforward manner. 

\begin{thm} \label{relOTM} \label{enumerateL} 
\begin{enumerate-(1)} 
\item 
Let $x,y\subseteq\text{On}$. Then $x$ is computable by some OTM-program with some ordinal parameter $\alpha$ in the oracle $y$ if and only if $x\in L[y]$.
\item 
There is a non-halting OTM-program $P$ such that for all $x\subseteq\omega$, $P^{x}$ enumerates $L[x]$ in the sense that for every set of ordinals $y\in L[x]$, the characteristic function of $y$ is written the tape at some time in the computation. 
\end{enumerate-(1)}
\end{thm} 
\begin{proof}
See \cite[Lemma 9]{CS}.
\end{proof}

A set $x\subseteq\omega$ is called ITTM-recognizable if and only if there is a program $P$ such that $P$ stops with output $\delta_{xy}$ with the oracle $y\subseteq\omega$ \cite{HaLe}. 
This generalizes to arbitrary sets of ordinals for OTMs. 
We will also use a relativized notion of recognizability. 
Relativized recognizability was introduced and studied for ITRMs (\cite{Ca}) and several more machine types \cite{Ca2}. 

\begin{definition}{\label{relrec}} 
Suppose that $x$ is a subset of an ordinal $\alpha$ and $y$ is a set of ordinals. 
\begin{enumerate-(i)} 
\item 
$x$ is \emph{recognizable} from $y$ from finitely many ordinal parameters $\gamma_0,\dots,\gamma_n$ if there is an OTM-program $P$ 
with the parameters $y$ and $\gamma_0,\dots,\gamma_n$ 
which halts for every subset $z$ of $\alpha$ with $\delta_{xz}$ in the first cell. 
\item 
$x$ is \emph{recognizable} from $y$ without parameters if we can choose $\langle \gamma_0,\dots,\gamma_n\rangle$ as the empty sequence. 
\item 
$x$ of an ordinal $\alpha$ is \emph{recognizable} from $y$ if it is recognizable from $y$ with some ordinal parameters. 
\end{enumerate-(i)} 
%
\end{definition} 


Via iterated Cantor pairing, we cam assume 
that the parameter is a single ordinal. Therefore we only consider single parameters from now on. 

We have the following simple characterisation of recognizability. 

\begin{lemma}\label{2.3}
A subset $x$ of an ordinal $\alpha$ is recognizable from a set of ordinals $y$ if 
and only if there is a $\Sigma_{1}$ formula $\varphi(u,v,w)$ and an ordinal $\beta$ so that $x$ is the unique subset $z$ of $\alpha$ so that $L[y,z]\models \varphi(y,z,\beta)$.
\end{lemma} 
\begin{proof} 
Suppose that $x$ is recognisable by a program $P_{e}$  in the parameters $y$ and $\beta$. Then the statement that this program halts on input $x$ with state $1$, is a $\Sigma_{1}$ statement true only of $x$ for this pair $y,\beta$ and is absolute. 

Conversely, suppose that $\varphi(y,z,\beta)$ is a $\Sigma_1$-formula for which $x$ is the unique solution for $z$ in models of the form $L[y,z]$. Suppose $\max\{
 \sup y,\alpha,\beta \} = \tau$. Let $P_{e}$ be the program with  parameters $y$,  the $V$-cardinal $\gamma = |\tau|^{+}$, and the $\beta$ in $\varphi$, that on input $z\subseteq\alpha$ checks if $L_{\gamma}[y,z]\models \varphi[y,z,\beta] $.
 If so then $z=x$ and it may halt with the state $1$. Otherwise  $z\neq x$ and it may halt with the state $0$. 
\end{proof}



We can define the following variations of recognizable sets and it is easy to show that each of these conditions is equivalent to recognizability. 
A subset $x$ of an ordinal $\alpha$ is \emph{semi-recognizable} if there is an OTM-program $P$ and an ordinal $\beta$ such that for all $y\subseteq\alpha$, $P^{y}(\beta)$ halts if and only if $x=y$. 
A subset $x$ of an ordinal $\alpha$ is \emph{anti-recognizable} if there is an OTM-program $P$ and an ordinal $\beta$ such that for all $y\subseteq\alpha$, $P^{x}(\beta)$ diverges if and only if $x=y$.

Moreover, recognizability is stable under computable equivalence for OTMs with ordinal parameters. 
This is equivalent to constructible equivalence, where $x,y$ are constructibly equivalent if $x\in L[y]$ and $y\in L[x]$. 

\begin{lemma}{\label{conequiv}} 
If sets of ordinals $x,y$ are constructibly equivalent, then  $x$ is recognizable implies $y$ is recognizable. 
\end{lemma}
\begin{proof} 
We use the characterisation of recognizable sets in Lemma \ref{2.3}. Let $ \varphi(v,\alpha)\in \Sigma_{1}$ have the unique solution $x$ in the parameter $\alpha$ in models of the form $L[z]$.  Let $x$ be the $\beta$'th set in $L[y]$ and $y$ the $\gamma$'th set in $L[x]$. We now check that $y$ is the unique set of ordinals so that $L[y]$ satisfies the $\Sigma_{1}$ formula  which states that the $\beta$'th set $z$ in $L[y]$ satisfies $\varphi(z,\alpha)$ and $y$ is the $\gamma$'th set in $L[z]$. 
\end{proof} 

In the previous lemma, it is not sufficient to assume that $y\in L[x]$. For example, suppose that $x=0^{\#}$ and $y$ is a Cohen real over $L$ which is constructible from $x$. 

However, recognizability from a fixed ordinal is not absolute. For instance, $0^{\#}$ is recognizable from $\omega_1$, if $0^{\#}$ exist, but it is not recognizable from any countable ordinal, as we see below. Therefore $0^{\#}$ is not recognizable from $\omega_1^V$ in any generic extension of $V$ in which $\omega_1^V$ is countable. We prove below that the recognizable closure for subsets of $\omega$, defined by iterating recognizability, is absolute if there is a measurable cardinal above a Woodin cardinal. 

A typical phenomenon for infinitary computations is the existence of sets of ordinals which are recognizable, but not computable. Following \cite{HaLe}, we call such sets \emph{lost melodies}. 

\begin{defini} 
Suppose that $x\subseteq\alpha$ is recognizable, but not computable. Then $x$ is a lost melody at $\alpha$. 
\end{defini} 

As mentioned in the introduction, we can iterate recognizability. 

\begin{defini}{\label{recclos}} 
Suppose that $x,y\subseteq \alpha$. 
\begin{enumerate-(i)} 
\item 
$x$ is an element of the \emph{recognizable closure $\CC^{\beta}(y)$ of $y$ in the parameter $\beta$} if there is a sequence $\langle x_i\mid i\leq n\rangle$ with $x_0=x$ and $x_n=y$ such that $x_i$ is recognizable from $x_{i+1}$ with the parameter $\beta$ for all $i<n$. 
\item 
$x$ is an element of the \emph{recognizable closure $\CC_\alpha^{\beta}(y)$ of $y$ for subsets of $\alpha$ in the parameter $\beta$} if we only allow subsets $x_i$ of $\alpha$ in the previous case. 
\item 
$x$ is an element of the \emph{recognizable closure $\CC_\alpha(y)$ of $y$ for subsets of $\alpha$} if $x\in \CC_\alpha^{\beta}(y)$ for some $\beta$. 
\item 
$x$ is an element of the \emph{recognizable closure} $\CC(y)$ of $y$ if $x\in \CC_\alpha(y)$ for some $\alpha$. 
\end{enumerate-(i)} 
We omit $y$ if $y=\emptyset$. 
\end{defini}


Note that the recognizable closure $\CC$ is closed under constructibility, i.e. a set of ordinals $x$ is in $\CC$ if $x\in L[x_0,\dots,x_n]$, where $x_0,\dots,x_n$ are sets of ordinals in $\CC$. 
In particular $\CC$ is closed under joins.

The iteration in the definition of the recognizable closure is necessary, since relativized recognizability is in general not transitive by  Lemma \ref{nontrans} below. 
However, two iteration steps suffice by the following argument. 

\begin{lemma}\label{twostep} 
Suppose that $x\in \CC_\alpha(z)$. 
Then there is some $y$ such that $y$ is recognizable from $z$ and 
$x$ is computable from $y$. 
\end{lemma} 
\begin{proof} 
Suppose that $\vec{x}=\langle x_i\mid i\leq n\rangle$ witnesses that $x\in \CC_\alpha(z)$ 
as in Definition \ref{recclos}. 
Suppose that $y$ is the join of $\vec{x}$. 
Then $y$ is recognizable and $x$ is computable from $y$. 
\end{proof}


The recognizable closure $\CC$ has the following absoluteness property.  

\begin{lemma} 
$\CC$ is $\Sigma^1_2$-correct in $V$. 
\end{lemma} 
\begin{proof} 
Suppose that $x\subseteq\omega$, $x\in \CC$ and $\varphi(x)$ is a $\Sigma^1_2$-formula which holds in $V$. 
Since $L[x]\subseteq \CC$, $\varphi(x)$ holds in $\CC$. 
If $x\in\CC$ and $\varphi(x)$ is a $\Sigma^1_2$-formula which holds in $\CC$. Then $\varphi(x)$ holds in $V$ by $\Pi^1_1$ absoluteness between $L[x,y]$ and $V$ for all $y\subseteq\omega$. 
If $\varphi(x)$ holds in $\CC$, then $\varphi(x)$ holds in $V$ by $\Pi^1_1$ absoluteness. 
\end{proof}



%
%

\section{Recognizable sets} 

In this section, we prove various facts about recognizable sets and lost melodies. 
We begin with recognizability without ordinals parameters for various infinite time machines. 
We then consider recognizability for fixed ordinals parameters, and finally recognizability for arbitrary ordinal parameters. 
We prove various facts about recognizability in $L$ and in generic extensions of $L$, and about the relationship between the recognizable closure and $\HOD$.

\subsection{Machines with fixed ordinal parameters} 

Recognizability is a general concept associated with models of infinitary computation that has been studied for Infinite Time Register Machines (ITRMs) \cite{ITRM} and Infinite Time Turing Machines (ITTMs) \cite{HaLe}. 
The definition of recognizability for ITRMs and ITTMs is completely analogous to Definition \ref{relrec} (\cite{HaLe},  \cite{Ca}). 
We denote by $\CC^M$ the recognizable closure for subsets of $\omega$ for a machine model $M$ without parameters. 

The notion of recognizable closure is remarkably stable, and in fact the recognizable closure for ITRMs, ITTMs and OTMs without ordinal parameters is the same. 
Moreover, the relation  that $x$ is in the recognizable closure of $y$ for ITRMs, ITTMs and OTMs without ordinal parameters coincides with $\Delta_{2}^{1}$-reducibility. 
Thus the recognizable closure for OTMs with ordinal parameters is generalization of $\Delta_{2}^{1}$-reducibility in at least two ways. First by admitting arbitrary ordinal parameters, second by admitting arbitrary sets of ordinals. 

We use the following notions to calculate the recognizable closure for ITTMs, ITRMs and OTMs. 

\begin{defini} 
Suppose that $x\subseteq\omega$ and $\alpha$ is an ordinal. 
\begin{enumerate-(i)} 
\item 
The ordinal $\alpha$ is \emph{$\Sigma_{1}$-fixed} if there is some $\Sigma_{1}$-statement $\varphi$ in the language of set theory such that $\alpha$ is minimal with $L_{\alpha}\models\varphi$. 
\item 
Let $\sigma$ denote the supremum of the $\Sigma_1$-fixed ordinals. 
\item 
The ordinal $\alpha$ is \emph{$\Sigma_{1}$-stable (in $x$)} if $L_{\alpha}\prec_{\Sigma_{1}}L$ ($L_{\alpha}[x]\prec_{\Sigma_{1}}L[x]$). 
\item 
Let $\sigma(\alpha)$ ($\sigma(\alpha,x)$) denote the least $\Sigma_{1}$-stable (in $x$) ordinal $\tau>\alpha$. 
\end{enumerate-(i)} 
\end{defini}

Note that the ordinal $\sigma$ in (ii)  is well known to be the least $\Sigma_{1}$-stable, here also notated $\sigma(0,0)$, and is countable. 
As the halting  of an OTM-program (without ordinal parameters) is a $\Sigma_{1}$-fact, and the search for the ordinal $\alpha$ fixed by some sentence $\varphi$ as in (i), is OTM-programmable, 
 $\sigma$ is equal to 
the supremum of such program halting times, \cite[Lemma 7]{CS}. 

\begin{lemma}\label{wrecclos}
$\CC^{\text{ITRM}}=\CC^{\text{ITTM}}=\CC^{\text{OTM}}=L_{\sigma}$. 
\end{lemma} 
\begin{proof} 
We first argue that all elements of $\CC^{\text{ITRM}}$, $\CC^{\text{ITTM}}$ and $\CC^{\text{OTM}}$ are elements of $L_{\sigma}$. 
Suppose that $P$ is a program for one of these machine types. 
The statement that 
there is some $x\subseteq\omega$ such that $P^{x}\downarrow=1$ is a $\Sigma_{1}$-statement. 
This is absolute between $L$ and $V$ by Shoenfield absoluteness 
and computations are absolute between transitive models of ZFC.
If $P$ recognizes some real number $x$,
then $L\models\exists{y}P^{y}\downarrow=1$ and hence $x\in L$. 
Since $\exists{y}P^{y}\downarrow=1$ is a $\Sigma_1$-statement, we have $x\in L_{\sigma}$. 
Since $L_{\sigma}$ is admissible, $\CC^M\subseteq L_{\sigma}$ for these machine models $M$ by the proof of Lemma \ref{twostep}. 

Suppose that $x\subseteq\omega$ and $x\in L_\sigma$. 
There is some $\Sigma_{1}$-fixed ordinal $\alpha<\sigma$ such that $x\in L_{\alpha}$. 
Suppose that $\varphi$ is a $\Sigma_1$-statement and $\alpha$ is minimal with $L_\alpha\vDash\varphi$. 
The $L_\alpha$ is its $\Sigma_1$-hull, hence its $\Sigma_1$-projectum is $\omega$. 
Then there is a surjection of $\omega$ onto $L_\alpha$ in $L_{\alpha+1}$ by acceptability. 
Let $c\subseteq\omega$ denote the $L$-least code for $L_\alpha$. 
Then $c$ is recognizable by any of these machines by checking whether for the least $\beta$ such that $\varphi$ holds in $L_\beta$, the oracle $z$
is the $L$-least code for $L_\beta$ 
 (see \cite{ITRM}). 
Moreover $x$ is Turing computable from $c$ and hence $x\in \CC^M$. 
\end{proof} 

A weaker version of ITRMs was introduced in \cite{wITRM}, now called unresetting or weak Infinite Time Register Machines (wITRMs). 

\begin{lemma} 
$\CC^{\text{wITRM}}=L_{\omega_{1}^{\text{CK}}}\cap P(\omega)$. 
\end{lemma} 
\begin{proof} 
Recognizability and computability coincide for wITRMs \cite{Ca2}. 
Moreover, the main result of \cite{wITRM} shows that the wITRM-computable reals coincide with the hyperarithmetical reals. 
\end{proof} 



The recognizable closure is closely connected with $\Delta_{2}^{1}$-degrees. 

\begin{defini}{\label{delta12degrees}}
Suppose that $x,y\subseteq\omega$ and $\alpha<\omega_1$. 
\begin{enumerate-(i)} 
\item 
$x$ is \emph{$\Delta_{2}^{1}$-reducible to $y$} ($x\leq_{\Delta_{2}^{1}}y$) if $x$ is $\Delta_{2}^{1}$ in $y$. 
\item 
$x\equiv_{\Delta_{2}^{1}}y$ if $x\leq_{\Delta_{2}^{1}}y$ and $y\leq_{\Delta_{2}^{1}}x$. 
\item 
The equivalence class $[x]_{\Delta_{2}^{1}}$ for $\Delta_{2}^{1}$-equivalence is called the \emph{$\Delta_{2}^{1}$-degree of $x$}.
\item 
If  $d_{1}=[x_{1}]_{\Delta_{2}^{1}}$ and $d_{2}=[x_{2}]_{\Delta_{2}^{1}}$ are $\Delta_{2}^{1}$-degrees with $x_{1}\leq_{\Delta_{2}^{1}}x_{2}$,
then we write $d_{1}\leq_{\Delta_{2}^{1}}d_{2}$. 
\item 
$x$ (a set $A$ of reals) is $\Delta^1_2(y,\alpha)$ ($\Sigma^1_2(y,\alpha)$) is there is a $\Delta^1_2$-definition ($\Sigma^1_2$-definition) in real parameters $y,z$ which defines $x$ (the set $A$) for all codes $z$ of $\alpha$. 
\item $x\leq_{\text{OTM}}y$ if $x$ is OTM-computable from $y$ without ordinal parameters. 
\end{enumerate-(i)} 
\end{defini} 

A part of the following lemma is known, but we could not find a proof in the literature. 
The equivalence of the first two conditions in the following lemma without parameters was independently proved in \cite{Da09}. 

\begin{lemma} \label{characterization of recognizable from alpha} 
Suppose that $x,y$ are subsets of $\omega$ and $\alpha$ is a countable ordinal. 
The following conditions are equivalent. 
\begin{enumerate-(a)} 
\item 
$x$ is computable from $y$ and $\alpha$.
\item 
$x$ is recognizable from $y$ and $\alpha$. 
\item 
$x\in \CC^\alpha_\omega(y)$. 
\item 
$x$ is $\Delta_1(y,\alpha)$. 
\item 
$x$ is $\Delta^1_2(y,\alpha)$. 
\item 
$\{x\}$ is $\Sigma^1_2(y,\alpha)$. 
\item 
$x\in L_{\sigma(y,\alpha)}[y]$. 
\end{enumerate-(a)} 
\end{lemma} 
\begin{proof} 
The condition (a) implies (b). 
The condition (b) implies (c) by the proof of Lemma \ref{twostep}. 
This also shows that (c) implies (d). 
It follows from the proof of \cite[Lemma 25.25]{MR1940513} that (d) implies (e). 
It is easy to see that (e) implies (f). 

Suppose that (f) holds. 
Suppose that $\{x\}$ is defined by a $\Sigma^1_2$-formula $\varphi(x,y,u)$, where $u$ is an arbitrary code for $\alpha$. 
Suppose that $G$ is $\mathrm{Col}(\omega,\alpha)$-generic over $V$ and that $u$ is a relation on $\omega$ in $L[G]$ which is isomorphic to $\alpha$. 
Let $\psi(z,y,v)$ denote the statement that there is a relation $v$ on $\omega$ such that $v$ is isomorphic to $u$ and $\varphi(z,y,v)$ holds. 

\begin{claim*} 
In $V[G]$ the real $x$ is the unique real $z$ with $\psi(z,y,u)$. 
\end{claim*} 
\begin{proof} 
Let $\chi(y,v)$ denote the statement that there is a real $z\neq x$ and a relation $w$ on $\omega$ such that $w$ is isomorphic to $v$ and $\varphi(z,y,w)$ holds. 
Suppose that $v$ is a code for $\alpha$ in $V$. 
The statement $\chi(y,v)$ 
is a $\Sigma^1_2$ statement in $y,v$ which is false in $V$ 
by the uniqueness of $x$. Hence $\chi(y,v)$ is false in $V[G]$ 
by Shoenfield absoluteness. 
Hence $x$ is the unique real $z$ with $\psi(z,y,v)$ in $V[G]$. Since the truth value of $\psi(z,y,v)$ is equal to that of $\psi(z,y,u)$ if $v$ is isomorphic to $u$, $x$ is the unique real $z$ with $\psi(z,y,u)$ in $V[G]$. 
\end{proof} 

Since $u\in L[G]$, there is a real $z$ in $L[G]$ with $\psi(z,y,u)$ by Shoenfield absoluteness. Then $z=x$ by the previous claim. 
Suppose that $H$ is $\mathrm{Col}(\omega,\alpha)$-generic over $V[G]$. Then $x\in L[G]\cap L[H]=L$. 
The $\Sigma^1_2$-statement in $x,y,\alpha$ given by (f) is equivalent to a $\Sigma_1$-statement $\theta(x,y,\alpha)$ by the proof of \cite[Lemma 25.25]{MR1940513}. 
Suppose that $\beta>\alpha$ is least such that there is a real $z$ in $L_\beta$ such that $\theta(z,y,\alpha)$ holds in $L_\beta$. 
Then $\beta<\sigma(y,\alpha)$ and hence $x\in L_{\sigma(y,\alpha)}[y]$. 
This implies (g). 

If (g) holds, then 
there is a $\Sigma_1$-formula $\theta(z,y,\alpha)$ such that $x\in L_\beta$ and $\beta>\alpha$ is least such that $\theta(x,y,\alpha)$ holds in $L_\beta$. 
Hence $x$ is computable from $y,\alpha$. This implies (a). 
\end{proof} 

We will see in the next section that recognizability from $\omega_1$ does not imply constructibility if $0^{\#}$ exists. 

The equivalence of the conditions (d) and (f) shows that OTM-reducibility coincides with $\Delta^1_2$-reducibility. 
Thus computability and recognizability generalize $\Delta_{2}^{1}$-reducibility in the following ways. 
Firstly, we allow arbitrary ordinal parameters instead of only countable ordinals, as in the previous lemma. 
Secondly, we can consider arbitrary sets of ordinals. 

The previous lemma can be generalized to uncountable ordinals as follows. 

\begin{lemma} 
Suppose that $x,y$ are subsets of $\omega$ and $\alpha$ is an ordinal. 
The following conditions are equivalent. 
\begin{enumerate-(a)} 
\item 
$x$ is computable from $y$ and $\alpha$.
\item 
$x$ is recognizable from $y$ and $\alpha$ in any generic extension in which $\alpha$ is countable. 
\item 
$x\in \CC^\alpha_\omega(y)$ in any $\mathrm{Col}(\omega,\alpha)$-generic extension (in any generic extension in which $\alpha$ is countable). 
\item 
$x$ is $\Delta_1(y,\alpha)$ in any $\mathrm{Col}(\omega,\alpha)$-generic extension (in any generic extension in which $\alpha$ is countable). 
\item 
$x$ is $\Delta^1_2(y,\alpha)$ in any $\mathrm{Col}(\omega,\alpha)$-generic extension (in any generic extension in which $\alpha$ is countable). 
\item 
$\{x\}$ is $\Sigma^1_2(y,\alpha)$ in any $\mathrm{Col}(\omega,\alpha)$-generic extension (in any generic extension in which $\alpha$ is countable). 
\item 
$x\in L_{\sigma(y,\alpha)}[y]$. 
\end{enumerate-(a)} 
\end{lemma} 
\begin{proof} 
This follows from Lemma \ref{characterization of recognizable from alpha} and the absoluteness of the conditions (c) and (d). 
\end{proof} 


The previous lemmas can be improved in $L$. 

\begin{lemma} 
Suppose that $\omega_1=\omega_1^L$ and $x,y\subseteq\omega$. 
Then $x\leq_{\text{OTM}} y$ holds if and only if $\sigma(x)\leq\sigma(y)$. 
\end{lemma} 
\begin{proof} 
Suppose that $x\leq_{\text{OTM}}y$. Then $x\in L_{\sigma(y)}[y]$ by Lemma \ref{characterization of recognizable from alpha}. 
Since $x$ is $\Sigma_1$-definable in $L_{\sigma(y)}[y]$, it follows that $\sigma(x)\leq\sigma(y)$. 

For the other implication, suppose that $\sigma ( x ) \leq \sigma ( y )$. 
Suppose that $x$ is the $\alpha$-th element of $L_{\sigma(y)}$. 
Since $\omega_1=\omega_1^L$, the ordinal $\alpha$ has a code which is computable from $y$ without ordinal parameters, as in the proof of Lemma \ref{wrecclos}. 
Hence $x$ is computable from $y$. 
\end{proof}

\begin{lemma}
Suppose that $V=L$ and $x,y\subseteq\omega$. 
If $x$ is recognizable from $y$ and $\alpha$, then $x$ is computable from $y$ and $\alpha$. 
\end{lemma}
\begin{proof}
Suppose that $P$ recognizes $x$ from $y$ and $\alpha$. 
We enumerate $L$ as in Theorem  \ref{enumerateL}. 
Whenever a new real number $z$ appears on the tape, we run $Q^{y,z}(\alpha)$ and return $z$ when $Q^{y,z}(\alpha)=1$. This computes $x$ from $y$ and $\alpha$. 
\end{proof}

Similar to the characterization of the reals in $L_{\sigma}$ in Lemma \ref{characterization of recognizable from alpha}, we ask how to characterize the reals in the transitive collapse of the $\Sigma_2$-hull in $L$. 

\begin{definition} \label{definition of eta} 
Let $\eta$ denote the ordinal such that $L_\eta$ is the transitive collapse of the $\Sigma_2$-hull in $L$. 
\end{definition} 

This is connected with \emph{eventually writable} sets of ordinals. 


\begin{definition} 
Suppose that $x\subseteq\gamma$. 
\begin{enumerate-(i)} 
\item 
$x$ is \emph{writable} if there is an OTM-program $P$ which halts with empty input, and $x$ is written on the initial segment of length $\gamma$ of the tape when the program halts. 
\item 
$x$ is \emph{eventually writable} if there is an OTM-program $P$ with empty input such that from some time onwards, the contents of the initial segment of the tape of length $\gamma$ is $x$. 
\item
$x$ is \emph{accidentally writable} if there is an OTM-program $P$ with empty input such that $x$ is written on the initial segment of length $\gamma$ of the tape at some time. 
\end{enumerate-(i)} 
\end{definition} 

\begin{lemma} \label{properties of the sigma2 hull} 
\begin{enumerate-(1)} 
\item 
Suppose that $z\in L$ and $L\vDash \exists x \forall y \varphi(x,y,z)$, where $\varphi$ is a $\Sigma_0$-formula. 
Then the $<_L$-least $z$ with $L\vDash \exists x \forall y \varphi(x,y,z)$ is $\Sigma_2$-definable over $L$. 
\item \label{characterization of the sigma2 hull} 
The $\Sigma_2$-hull in $L$ with respect to the canonical Skolem functions is equal to 
the set of $\Sigma_2$-definable elements of $L$. 
\item 
$\eta$ is the least ordinal $\gamma$ such that $L_\gamma$ and $L$ have the same $\Sigma_2$-theory. 
\end{enumerate-(1)} 
\end{lemma} 
\begin{proof} 
The $\Sigma_2$-definition of $z$ in $L$ states that there is some $x$ such that for all $\alpha$ with $z,x\in L_\alpha$, 
$L_\alpha\vDash \forall y \varphi(x,y,z)$ and there are no $\bar{z},\bar{x}\in L$ with $(\bar{z},\bar{x})<_{\mathrm{lex}}(z,x)$ such that $L_\alpha\vDash \forall y \varphi(\bar{x},y,\bar{z})$. 
Here $<_{\mathrm{lex}}$ is the lexicographical order with respect to $<_{L}$. 

For the second claim, note that the set of $\Sigma_2$-definable elements of $L$ is contained in the $\Sigma_2$-Skolem hull in $L$. 
For the other inclusion, we prove the following claim. 

\begin{claim*} 
Suppose that $\varphi$ is a $\Sigma_0$-formula and $a$ is $\Sigma_2$-definable over $L$. 
Suppose that $L\vDash \exists x \exists y\forall z \varphi(x,y,z,a)$. 
Then there is a $\Sigma_2$-definable element $x$ of $L$ with $L\vDash \exists y\forall z \varphi(x,y,z,a)$. 
\end{claim*} 
\begin{proof} 
Suppose that $\exists b\forall c \psi(a,b,c)$ defines $a$ over $L$. 
Then the $<_L$-least $x$ which satisfies the formula $\exists a,b,y \forall c,z (\varphi(x,y,z,a)\wedge \psi(a,b,c))$ is $\Sigma_2$-definable over $L$ by the first claim. 
\end{proof} 
This implies that every element of the $\Sigma_2$-hull in $L$ is $\Sigma_2$-definable. 

For the third claim, suppose that $\varphi(x,i)$ is a universal $\Sigma_2$-formula. 
We define $i<j$ by $\exists \alpha,\beta\ \varphi(i,\alpha)\wedge \varphi(j,\beta)\wedge \alpha<\beta$. 
Since every element in $L_\eta$ is $\Sigma_2$-definable, this defines a well-order with order type $\eta$ in $L_\eta$. 
Since this formula defines the same well-order in $L_\gamma$ and in $L_\eta$, we have $\gamma=\eta$. 
\end{proof} 

We defined $L_\eta$ in Definition \ref{definition of eta}. 

\begin{lemma} \label{eventually writable} 
The eventually writable subsets of $\omega$ are exactly the subsets of $\omega$ in $L_\eta$. 
\end{lemma} 
\begin{proof} 
Every eventually writable real is $\Sigma_2$-definable over $L$. 

Suppose that $x\subseteq\omega$ is an element of $L_\eta$. Then $x$ is $\Sigma_2$-definable over $L$ by Claim \ref{characterization of the sigma2 hull} in Lemma \ref{properties of the sigma2 hull}. 
Suppose that $\exists y \forall z \varphi(x,y,z)$ defines $x$ over $L$, where $\varphi$ is a $\Sigma_0$-formula. 
we consider an OTM which 
enumerates the ordinals $\alpha$. 
For each $\alpha$, we consider the $<_{L}$-least pair $(x,y)\in L_\alpha$ such that $L_\alpha\vDash \forall z \varphi(x,y,z)$ and write $x$ on the initial segment of the tape. 
This will eventually write $x$. 
\end{proof}

\begin{lemma} 
\begin{enumerate-(1)} 
\item 
The writable reals are exactly the reals in $L_{\sigma}$. 
\item 
The eventually writable reals are exactly the reals in $L_{\eta}$. 
\item 
The accidentally writable reals are exactly the reals in $L$. 
\end{enumerate-(1)} 
\end{lemma} 
\begin{proof} 
This follows from Lemma \ref{characterization of recognizable from alpha} and Lemma \ref{eventually writable}. 
\end{proof} 

This is analogous to ITTMs, where the writable, eventually writable and accidentally writable reals exactly correspond to levels of $L$.

\subsection{Different ordinals} 



In this section, we compare the recognizability strength of OTMs for different ordinal parameters. A first observation is that the recognizability strength does not depend monotonically on the parameter. 

\begin{lemma}{\label{nonmon}}
There is a real $x$ in $L$ which is recognizable from some ordinal $\alpha$ that is countable in $L$, but not from $\omega_1$. 
\end{lemma}
\begin{proof} 
Let $c_{\alpha}$ denote the $L$-least real which codes $\alpha$ for each $\alpha<\omega_1$. 
Let $P_i$ denote the program with index $i$. 
We define $f:\omega\rightarrow\omega_{1}^{L}$ by $f(i)=\alpha$ if $\alpha<\omega_{1}^{L}$ is minimal such that $P_{i}^{c_{\alpha}}(\omega_{1})=1$ and $f(i)=0$ if there is no such $\alpha$. 
Then $\mathrm{ran}(f)$ contains all ordinals $\alpha<\omega_{1}^{L}$ such that $c_\alpha$ is recognizable in the parameter $\omega_{1}$. 
Since $f\in L$, $\mathrm{ran}(f)$ is countable in $L$ 
and there is some $\gamma\in\omega_{1}^{L}\setminus \mathrm{ran}(f)$. 
Then $c_{\gamma}$ is 
not recognizable in $\omega_{1}$.
But $c_{\gamma}$ can be recognized from $\gamma$ by computing the $L$-least real coding $\gamma$ and comparing this with the input. 
\end{proof} 

On the other hand, if $0^{\#}$ exists, then it is recognizable from $\omega_1$ by \cite[Theorem 4.2]{Ca2}, while it is not recognizable from any countable ordinal. 

\begin{lemma}{\label{0sharprecog}}
Suppose that $x\subseteq\omega$ and $x^{\#}$ exists. Then $x^{\#}$ is recognizable from $x$. 
Moreover $x^{\#}$ is recognizable from $x$ and $\alpha$ if and only if $\omega_1\leq \alpha$.
\end{lemma} 
\begin{proof} 
Suppose that $\alpha\geq\omega_1$. 
The relation $x^{\#}=y$ is $\Pi^1_2$ and hence absolute between $L_\alpha[x,y]$ and $V$ \cite[Theorem 25.20]{MR1940513}. 
Given a real $y$, we check the definition of $x^{\#}$ for $y$ in $L_\alpha[x,y]$. 
This shows that $x^{\#}$ is recognizable from $x$ and $\alpha$.

Since $x^{\#}$ exists and hence $\omega_1$ is inaccessible in $L[x]$, 
Lemma \ref{characterization of recognizable from alpha} shows tha 
$x^{\#}$ is not recognizable from $x$ and any countable ordinal. 
\end{proof} 

The previous lemma implies that $x^{\#}$ is recognizable if $x$ is recognizable. 

To show that recognizability from $\omega_\alpha$ is different in $L$ for different countable ordinals $\alpha$ in $L$, we use the following jump operator. 
Similar jump operators for different machine models are studied in \cite{Ca5}.

\begin{definition} 
Suppose that $\alpha\in \Ord$. 
The recognizable jump for $\alpha$ is the set $J^\alpha\subseteq\omega$ of all codes for programs $P$ such that $P^{\alpha}$ recognizes some $x\subseteq\omega$. 
\end{definition} 

\begin{lemma}{\label{rjnonrec}} 
$J^{\alpha}$ is not recognizable in $\alpha$ for any $\alpha$. 
\end{lemma} 
\begin{proof} 
Suppose that $J^\alpha$  can be recognized from $\alpha$. 
Let $x_n$ denote the real recognized by $P_n(\alpha)$ is $n\in J^\alpha$ and let $x_n=0$ if $n\notin J^\alpha$. 
We can diagonalize and construct a real $y$ such that $y\neq x_n$ for all $n$ and $y$ is recognizable from $\alpha$. 
\end{proof}

\begin{lemma} 
Supppose that $V=L$ and $\alpha<\omega_1$.
Then there is a real $x$ which is recognizable from $\omega_{\alpha+1}$ but not from $\omega_\alpha$. 
\end{lemma} 
\begin{proof} 
We show that $J_{\omega_\alpha}$ can be recognized from $\omega_{\alpha+1}$. 
Since $V=L$, the cardinal $\omega_\alpha$ can be computed from $\omega_{\alpha+1}$ as the largest cardinal in $L_{\omega_{\alpha+1}}$. 
This implies that the halting problem for computations with the parameter $\omega_{\alpha}$ is computable from the parameter $\omega_{\alpha+1}$ and this implies the claim. 
Moreover $J_{\omega_{\alpha}}$ is not recognizable from $\omega_{\alpha}$ by Lemma \ref{rjnonrec}. 
\end{proof} 

This fails if $x^{\#}$ exists for every real $x$. 

\begin{lemma} 
Suppose that $x^\#$ exists for every real $x$. 
Suppose that $\kappa$ and $\lambda$ are uncountable cardinals. 
If $P$ recognizes $x$ from $\kappa$, then $P$ recognizes $x$ from $\lambda$. 
\end{lemma} 
\begin{proof}
Suppose that $P^x(\kappa)$ halts with output $1$. 
Since $\kappa$ and $\lambda$ are indiscernible over $L[x]$, $P^x(\lambda)$ halts with output $1$. 
We claim that $P^x(\lambda)$ halts with output $0$ if $x\neq y$. 
Otherwise $P^x(\lambda)$ halts with output $0$ or diverges.  
Since $\kappa$ and $\lambda$ are indiscernible over $L[x]$, the same holds for $P^y(\lambda)$, contradicting the assumption. 
\end{proof} 



\subsection{Arbitrary ordinals} 




In this section, we consider the recognizable closure with arbitrary ordinal parameters. 
We first give several alternative definitions of the recognizable closure. 

\begin{definition} 
The logic $L_{Ord,0}^{(\beta)}$ is the closure under infinitary conjunctions and disjunctions of the atomic statements 
$\alpha\in \dot{x}$ and their negations for $\alpha<\beta$. 
\end{definition} 

The sentences in $L_{Ord,0}^{(\beta)}$ are interpreted as descriptions of a subset $x$ of $\beta$. 
This is related to the notion of implicitly definable sets from \cite{HaLeah}. 

\begin{definition} 
Suppose that $\alpha$ is an ordinal and $x$ is a subset of $\alpha$. 
\begin{enumerate-(i)} 
\item 
$x$ is \emph{implicitly definable over $L$} if there is a formula $\varphi(z,\beta)$ and an ordinal $\beta$ such that $x$ is the unique subset $z$ of $\alpha$ such that $\langle L,\in, z\rangle\vDash \phi(z,\beta)$. 
\item 
$x$ is \emph{$(L_{Ord,0}\cap L)$-definable} if there is an $L_{Ord,0}^{(\alpha)}$-formula $\varphi(z,\beta)$ in $L$ and an ordinal $\beta$ such that $x$ is the unique subset $z$ of $\alpha$ such that $\varphi(z,\beta)$ holds. 
\item 
$x$ is \emph{$S^L$-definable} if there is a formula $\varphi(z,\beta)$ and an ordinal $\beta$ such that $x$ is the unique subset $z$ of $\alpha$ with $L[z]\vDash \phi(z,\beta)$. 
\end{enumerate-(i)} 
\end{definition} 

The $S^L$-definable sets are a variant of the $S^L$-definable sets in \cite[Section 1]{LuSc}. 
The following result connects recognizable sets with implicitly definable sets. 
We would like to thank Gunter Fuchs for suggesting that there might be a relationship between these notions. 

\begin{lemma} 
Suppose that $\alpha$ is an ordinal and $x$ is a subset of $\alpha$. 
The following conditions are equivalent. 
\begin{enumerate-(i)} 
\item 
$x$ is constructible from a recognizable subset $y$ of some ordinal $\beta$. 
\item 
$x$ is constructible from a subset $y$ of some ordinal $\beta$ which is $S^L$-definable from some ordinal by a $\Sigma_1$-formula. 
\item 
$x$ is constructible from a subset $y$ of some ordinal $\beta$ which is $S^L$-definable from some ordinal. 
\item 
$x$ is constructible from some subset $y$ of an ordinal $\beta$ which is implicitly definable over $L$. 
\item 
$x$ is constructible from some $(L_{Ord,0}\cap L)$-definable subset $y$ of an ordinal $\beta$. 
\end{enumerate-(i)} 
\end{lemma} 
\begin{proof} 
Suppose that (i) holds. 
Suppose that the $\Sigma_1$-formula $\varphi$ states that there is a computation which halts with end state $1$ for the input $y$. This implies (ii). 
the condition (ii) implies (iii). 

Suppose that (iii) holds. 
Suppose that $\varphi$ is a $\Sigma_n$-formula and $\gamma>\alpha,\beta$ is an ordinal with $V_\gamma \prec_{\Sigma_{n+1}} V$. 
Then $L_{\gamma}[z]\prec_{\Sigma_n} L[z]$ and hence $L[z]\vDash \varphi(z,\beta) \Leftrightarrow L_\gamma [z]\vDash \varphi(z,\beta)$ for all $z
\subseteq\omega$ and $\beta<\gamma$. 

\begin{claim*} 
There is a set $A$ of ordinals such that $A$ is implicitly definable over $L$ and $L_\gamma[y]$ is constructible from $A$. 
\end{claim*} 
\begin{proof} 
We consider the G\"odel functions in \cite[Definition 13.6]{MR1940513} with three additional functions $H_0(u,v)=\langle u,v\rangle$, 
$H_1(u)=u$ and $H_2(u)=u\cap y$. 
Instead of the $L$-hierarchy over $y$, we consider  the hierarchy of sets $M_\alpha$, where $M_{\alpha+1}$ is defined by closing under the G\"odel functions with the additional functions. 
%
This induces a canonical wellorder on $L[y]$. 

We can code the levels of the hierarchy in the intervals between a strictly increasing sequence $\langle \delta_\alpha\mid \alpha<\gamma\rangle$ such that for each $\alpha<\gamma$, there is a canonical bijection between the interval $[\delta_\alpha,\delta_{\alpha+1})$ and all applications of the G\"odel functions to ordinals below $\delta_\alpha$. Let $\delta=\sup_{\alpha<\gamma}\delta_\alpha$. 
Note that every element of $L_{\gamma}[y]$ has many representations. 
We consider  
$\langle \delta_\alpha\mid \alpha<\gamma\rangle$, 
the pointwise images of the equality and element relations of $L_\gamma[y]$ and 
the image of the canonical well-order of $L_\gamma[y]$ in $\delta$. 
Let $A$ code these sets modulo G\"odel pairing. Note that G\"odel pairing is computable. 
We consider the implicit definition of $A$ which states that the sets coded by $A$ follow the construction of the hierarchy and that $\varphi(z,
\beta)$ holds in the structure code by $A$. 
Hence $A$ is implicitly definable over $L$. 
\end{proof} 
Therefore (iii) implies (iv). 

Suppose that (iv) holds. 
Suppose that $\varphi$ is a $\Sigma_n$-formula and $\gamma$ is an ordinal with $V_\gamma \prec_{\Sigma_{n+1}} V$ and $\gamma>\alpha,\beta$. 
Then $L_{\gamma}[z]\prec_{\Sigma_n} L[z]$ and hence $L[z]\vDash \varphi(z,\beta) \Leftrightarrow L_\gamma [z]\vDash \varphi(z,\beta)$ for all $z\in V_\gamma$ and $\beta<\gamma$. 
Then $y$ is an $(L_{\gamma,0}\cap L)$-definable subset of $\alpha$. 
Note that this is the translation of a first-order formula and hence its depth is finite. 
This implies (v). 

Suppose that (v) holds. Since validity of an $(L_{Ord,0}^{(\alpha)}\cap L)$-formula for $z$ can be calculated in $L[z]$, the set $y$ is recognizable. 
%
\end{proof} 

%

While the proof shows that every subset of an ordinal which is implicitly definable over $L$ is recognizable, the converse is open. 

\begin{question} 
\begin{enumerate} 
\item 
Is is consistent that there is a recognizable set of ordinals which is not implicitly definable? 
\item 
(see \cite[after Corollary 8]{HaLeah}) 
Is there a real $y$ which is Turing equivalent to $0^{\#}$ and which  is implicitly definable? 
\end{enumerate} 
\end{question} 

We mentioned $0^{\sharp}$ as an example of a lost melody. However, large cardinal assumptions are not necessary for the existence of lost melodies. 

\begin{thm}{\label{genericlostmelody}} 
There is a set-generic extension of $L$ by a real such that in the extension, every set of ordinals is recognizable. 
\end{thm} 
\begin{proof} 
There is a c.c.c. subforcing of Sacks forcing in $L$ which adds a $\Pi^1_2$-definable minimal real $x$ over $L$ \cite{Je1}. 
Since $x$ is $\Pi^1_2$-definable, it is recognizable from $\omega_1$ by Shoenfield absoluteness, as in the proof of Theorem \ref{0sharprecog}. 
Clearly, every constructible real is computable and hence recognizable. 
By minimality we have $x\in L[y]$ for every real $y\in L[x]\setminus L$.
Hence all non-constructible elements of the generic extension are constructibly equivalent to each other and in particular to $x$. 
Since $x$ is recognizable, it follows from Theorem \ref{conequiv}  that every real $y\in L[x]$ is recognizable. 

We now argue that the Jensen real is minimal. 
Jensen \cite[Lemma 11]{Je1} showed that the Jensen real is minimal for reals. 
We would like to thank Vladimir Kanovei for the following argument. 

\begin{claim*} 
Suppose that $G$ is Jensen generic over $L$. 
Suppose that $X\subseteq\kappa$ is a set of ordinals in $L[G]\setminus L$. 
Then there is a real $y$ such that $L[X]=L[y]$. 
\end{claim*} 
\begin{proof} 
Let $\mathbb{J}$ denote Jensen forcing. 
Suppose that $\sigma$ is a $\mathbb{J}$-name for a set of ordinals, and $1_{\mathbb{J}}\Vdash \sigma \notin L$. 
Then for any pair $S,T$ of conditions in $\mathbb{J}$, there is a pair of conditions $\bar{S},\bar{T}$ with $\bar{S}\leq S$, $\bar{T}\leq T$ and an ordinal $\alpha_{S,T}$ such that $\bar{S} \Vdash \alpha_{S,T} \in \sigma$, while  $\bar{T}\Vdash \alpha_{S,T} \notin \sigma$. 
Suppose that $A$ is a maximal antichain in $\mathbb{J}^2$. 
The forcing $\mathbb{J}^2$ is c.c.c. by [Jensen: Definable sets of minimal degree, Lemma 6], and in fact $\mathbb{J}^n$ is c.c.c. for any $n$. 
Hence $A$ is countable. Let $B$ denote the set of $\alpha_{S,T}$ for all $(S,T)\in A$. 

Then for any pair $S,T$ of conditions in $\mathbb{J}$, there is a pair $\bar{S},\bar{T}$ with $\bar{S}\leq S$ and $\bar{T}\leq T$ and an ordinal $\alpha\in B$ such that $\bar{S} \Vdash \alpha \in \sigma$, while  $\bar{T}\Vdash \alpha \notin \sigma$. 
Let $\tau$ denote a $\mathbb{J}$-name for $\sigma^G\cap B$, where $G$ is $\mathbb{J}$-generic over $L$. 
Since $B$ is countable in $L$, there is a real $y\in L[G]$ such that $L[\tau^G]=L[y]$ and hence $y\notin L$. 
Since the Jensen real is minimal for reals, $L[x]=L[y]$. 
Since $\tau^G\in L[\sigma^G]$, we have $L[\sigma^G]=L[x]$. 
\end{proof} 
This completes the proof. 
\end{proof} 

%
%

Theorem \ref{genericlostmelody} and the non-existence of lost melodies in $L$ show that it is undecidable in $\ZFC$ whether there are lost melodies for OTMs. 

We will consider recognizable sets in extensions by homogeneous forcings in the following proofs. 

\begin{definition} 
A forcing $\mathbb{P}$ is \emph{homogeneous} if for all conditions $p,q\in \mathbb{P}$, there is an automorphism $\pi\colon \mathbb{P}\rightarrow \mathbb{P}$ 
with $p\parallel \pi(q)$ (i.e. $p$ and $\pi(q)$ are compatible). 
\end{definition} 

We will use the fact that homogeneous forcings do not add recognizable sets. This was proved indepently in \cite{Da09} for Cohen forcing. 

\begin{lemma}\label{weakly homogeneous} 
We work in $\mathsf{ZF}$. 
Suppose that $\mathbb{P}$ is a homogenous forcing and $G$ is $\mathbb{P}$-generic over $V$. 
Suppose that $\mu\in \mathrm{Ord}$ and $x$ is a recognizable subset of $\mu$ in $V[G]$. 
If $x$ is recognizable, then $x\in V$. 
\end{lemma}
\begin{proof}
Suppose that $p\in \mathbb{P}$ forces that $x$ is recognized by a program $P$ from some ordinal $\gamma$. 
Further suppose that $x\notin V$ and that $\dot{x}$ is a $\mathbb{P}$-name for $x$. 
Then $p$ does not decide $\dot{x}$. 
Then there is some $\alpha<\mu$ and conditions $q,r\leq p$ such that $q\Vdash \dot{x}(\alpha)=0$ and $r\Vdash \dot{x}(\alpha)=1$. 
Let $\pi$ be an automorphism of $\mathbb{P}$ such that $q\parallel \pi(r)$
and suppose that $s\leq \pi(q),r$. 

Now suppose that $G$ is $\mathbb{P}$-generic over $V$ with $s\in G$. 
Since $s$ forces that $\dot{x}$ is recognized by $P$ in the parameter $\gamma$, we have 
$q\Vdash{P}^{\dot{x}}({\gamma})\downarrow $ and $\pi(r)\Vdash {P}^{\dot{x}}({\gamma})\downarrow$. 
Since $r$ forces that $\dot{x}$ is recognized by $P$ in the parameter $\gamma$, we have $\pi(r)\Vdash {P}^{\pi(\dot{x})}({\gamma})\downarrow$. 
We have $q\Vdash \dot{x}(\alpha)=0$ and $\pi(r)\Vdash \pi(\dot{x})(\alpha)=1$. 
Let $x=\dot{x}^{G}$ and $y=\pi(\dot{x})^G$. 
We work in $V[G]$. 
Since $q\in G$, $P$ recognizes $x$ from $\gamma$. 
Since $\pi(r)\in G$, $P$ recognizes $y$ from $\gamma$. 
However $x\neq y$, contradicting the uniqueness of $x$. 
\end{proof} 

We can now show that it is consistent with ZFC that relativized recognizability is not transitive. 

\begin{lemma}\label{nontrans} 
Assume that $V=L[0^{\sharp}]$. Then there are real numbers $x,y,z$ such that $x$ is recognizable from $y$, $y$ is recognizable from $z$, and $x$ is not recognizable from $z$. 
\end{lemma} 
\begin{proof} 
Suppose that $\gamma$ is the least ordinal with $\gamma>\omega_{1}^{L}$ such that $L_{\gamma+1}[0^{\sharp}]\setminus L_{\gamma}[0^{\sharp}]$ contains a real. Let $\delta:=\gamma+3$. 
There is a Cohen real $x$ over $L$ in $L_\delta$ and this is not recognizable by Lemma \ref{weakly homogeneous}. 
Let $c$ denote the $L[0^{\#}]$-least real in $L[0^{\#}]$ which codes $L_\delta$. 
Then $x$ is recognizable from $c$. 

On the other hand $c$ is recognizable from $0^{\sharp}$ in the parameter $\delta$. 
To determine whether a real $z$ is equal to $c$, we first check whether $z$ codes $L_{\alpha}[0^{\sharp}]$. 
In this case, we compute from $z$ a code $d$ for $L_{\delta+1}[0^{\sharp}]$ and determine whether $z$ is equal to the least element of the structure coded by $d$ which codes $L_\delta[0^{\#}]$. 
Moreover $0^{\sharp}$ is recognizable from $0$ by Lemma \ref{0sharprecog}. 
Since $x$ is not recognizable from $0$, this is a counterexample to transitivity. 
\end{proof} 


\subsection{The recognizable closure and $\mathrm{HOD}$} 
In this section, we show that every set of ordinals which is generic over $L$ can be coded into the recognizable closure $\CC$ in a further generic extension. 
The recognizable closure is contained in $\HOD$. We further show that there is a generic extension of $L$ in which the recognizable closure is not equal to $\HOD$. 

\begin{lemma} 
Suppose that $G$ is $\mathbb{P}$-generic over $\HOD$ and $X\in L[G]$ is a set of ordinals. 
Then there is a cofinality-preserving forcing $\mathbb{Q}$ in $L[G]$ such that for every $\mathbb{Q}$-generic filter $H$ over $L[G]$, $X$ is computable from a recognizable set in $L[G,H]$. 
\end{lemma} 
\begin{proof} 
Let $\lambda=|\mathbb{P}|^{+}$. 
Then $\mathbb{P}$ adds no Cohen subset to any regular cardinal $\nu\geq\lambda$, since otherwise some $p\in\mathbb{P}$ decides unboundedly many $\alpha<\nu$. 
Suppose that $\mu$ is a cardinal with $X\subseteq \mu$. 

The forcing $\mathbb{Q}\in L[G]$ is an iteration with full support. 
We define the iteration separately in the intervals between the ordinals in a strictly increasing sequence $\langle \mu_n\mid n\in\omega\rangle$ of length $\omega$ which is defined as follows. 

Let $\mu_0=\mu$. In the first interval $[0,\mu_0)$, we add a Cohen subset to $[\lambda^{+(2\cdot \alpha)},\lambda^{+(2\cdot\alpha+1)})$ if $\alpha\in X_0:=X$ and a Cohen subset to $[\lambda^{+(2\cdot\alpha+1)},\lambda^{+(2\cdot\alpha+2)})$ if $\alpha\notin X_0$ for $\alpha<\mu_0$. This iteration of length $\mu_0$ adds a subset $X_1$ of $\mu_1:=\lambda^{+\mu}$. 
We similarly define $\mu_{n+1}$ from $\mu_n$ and  $X_{n+1}$ from $X_n$ in the interval $[\mu_n,\mu_{n+1})$ for all $n$. 
Let $Y=\bigcup_{n\in\omega} X_n$ and 
let $H$ denote the generic filter over $L[G]$ defined by this sequence. Then $L[G,Y]=L[G,H]$. 

This iteration adds Cohen subsets over $L$ only to the successor cardinals $\nu\geq\lambda$ specified in the iteration. 

\begin{claim*} 
$Y$ is recognizable in $L[G,H]$. 
\end{claim*} 
\begin{proof} 
Suppose that $\bar{Y}$ is a subset of $\sup_{n\in\omega} \mu_n$. 
We can determine in $L[\bar{Y}]$ whether $\bar{Y}$ is consistent with the coding described above. 
Suppose that $\bar{Y}$ is consistent with the coding. 
Then $\bar{Y}\cap [\mu_n,\mu_{n+1})=Y\cap [\mu_n,\mu_{n+1})$, since this set is determined by the set of successor cardinals $\nu\geq \mu_{n+1}$ such that there is a Cohen subset of $\nu$ over $L$. 
Hence $\bar{Y}= Y$. 
\end{proof} 
This completes the proof. 
\end{proof} 


The next lemma shows that it is consistent that the recognizable closure is strictly contained in $\HOD$. Note that every set of ordinals in $\CC$ is $\Delta_1$-definable from an ordinal. 

\begin{lemma} 
Let $\mathbb{P}$ denote Cohen forcing and suppose that $\dot{x}$ is a $\mathbb{P}$-name for the $\mathbb{P}$-generic real. 
Suppose that $\dot{\mathbb{Q}}$ is a $\mathbb{P}$-name for the finite support product $\prod_{n\in x} \mathrm{Add}(\omega_n,1)$, where $x$ is the interpretation of $\dot{x}$. 
Suppose that $x$ is $\mathbb{P}$-generic over $L$ and that $G$ is $\dot{\mathbb{Q}}^x$-generic over $L[x]$. 
Then in $L[x,G]$, $x$ is $\Sigma_1$-definable from an ordinal, but not $\Pi_1$-definable from an ordinal. 
Hence $\CC\subsetneq HOD$ in $L[x,G]$. 
\end{lemma} 
\begin{proof} 
Suppose that $\varphi$ is a $\Pi_1$-formula, $\delta\in \mathrm{Ord}$ and $(p,\dot{q})$ is a condition which forces $\forall n (n \in \dot{x}\Leftrightarrow  \varphi(n,\delta))$ 
Suppose that $s\subseteq\omega$ is finite and $\mathrm{dom}(p)\subseteq s$. 
We can assume that $p\Vdash \mathrm{supp}(\dot{q})\subseteq s$. 

Suppose that $(x,G)$ is $\mathbb{P}*\dot{\mathbb{Q}}$-generic over $L$ below $(p,\dot{q})$ and $G=\prod_{i\in x} G_i$. 
Suppose that $n\in \omega\setminus (x\cup s)$ and $y=x\cup \{n\}$. 
Suppose that $G_n$ is $\mathrm{Add}(\omega_n,1)$-generic over $L[x,G]$. 
Suppose that $\dot{q}^x$, $\dot{q}^y$ are given by the sequences $\langle q_i^x\mid i\in s\cap x\rangle$ and $\langle q_i^y\mid i\in s\cap x\rangle$, where $q_i^x, q_i^y\in \mathrm{Add}(\omega_i,1)$. 
Suppose that $\pi_i\colon \mathrm{Add}(\omega_i,1)\rightarrow \mathrm{Add}(\omega_i,1)$ is an automorphism such that $\pi_i(p_i)$ is compatible with $q_i$ for all $i\in s\cap x$. 
Let $H$ denote the $\dot{\mathbb{Q}}^y$-generic filter over $L[y]$ which is equivalent to $\prod_{i\in s\cap x}\pi[G_i]\times \prod_{i\in x\setminus s}G_i \times G_n$ modulo the order of the indices in the product. 
Then $(y,H)$ is $\mathbb{P}*\dot{\mathbb{Q}}$-generic over $L$ below $(p,\dot{q})$ and $L[x,G,G_n]=L[y,H]$. 
Then $\varphi(n,\delta)$ holds in $L[y,H]$ by $\Sigma_1$ upwards absoluteness, contradicting the assumption on $(p,\dot{q})$. 
\end{proof} 


\begin{remark} 
Suppose that $\langle \kappa_i \mid i<\omega \rangle$ is a strictly increasing sequence of measurable cardinals and $\vec{\mu}=\langle \mu_i\mid i<\omega\rangle$ is a sequence of normal ultrafilters with $\mathrm{crit}(\mu_i)=\kappa_i$ for all $i$. 
Suppose that $V=L[\vec{\mu}]$. 
Then $\vec{\mu}$ is not coded by any set in $\CC$. 
In particular $\CC\subsetneq HOD$. 
\end{remark} 
\begin{proof} 
As in Lemma \ref{non-recognizable structure} below.  
\end{proof} 

\begin{question} 
Is it consistent that $\CC=L$ and $\CC\subsetneq\HOD$? 
\end{question}

\section{Recognizable sets and $M_1$}

We aim to define the recognizable closure for subsets of an ordinal $\alpha$ from 
$M^{\infty}$. 
$M^{\infty}$ is defined as follows, assuming that $M_1^{\#}$ exists. 
$M_1^{\#}$ is defined in \cite[Section 5.1]{Schimmerling} and \cite[page 1660]{Ste}. 
We work with premice of the form $M=(J_\alpha,\in,\vec{E})$ as defined in \cite{MiSte}. 

\begin{definition} 
$M_1^{\#}(x)$ denotes a sound $\omega_1$-iterable $x$-premouse $M$ with projectum $\omega$ and with an external measure with critical point $\kappa$ above a Woodin cardinal in $M$, if this exists. 
\end{definition} 

It is known that $M_1^{\#}$ is unique if it is $\omega_1+1$-iterable or if $M_1^{\#}(z)$ exists for every real $z$ \cite[Lemma 2.41]{Schlicht}. In the following we will make one of these assumptions. 

\begin{definition} 
\begin{enumerate-(i)} 
\item 
$(M_1^{\#})^{\alpha}$ denotes the $\alpha$-th iterate of $M_1^{\#}$ by the external measure and $\kappa_\alpha$ the image of $\kappa_0=\kappa$ in $(M_1^{\#})^{\alpha}$. 
\item 
$M_1:=\bigcup_{\alpha\in \mathrm{Ord}} (M_1^{\#})^{\alpha} |\kappa_\alpha$. 
\item 
$M^{\alpha}$ denotes the $\alpha$-th iterate of $M_1^{\#}$ by the unique normal measure on its least measurable cardinal $\mu$ and $\mu_\alpha$ the image of $\mu_0=\mu$ in $M^{\alpha}$. 
\item 
$M^{\infty}:=\bigcup_{\alpha\in \mathrm{Ord}} M^{\alpha} |\mu_\alpha$. 
\end{enumerate-(i)} 
\end{definition} 

If there is a measurable cardinal above a Woodin cardinal, then $M_1^{\#}$ exists \cite{MiSte, Ste1} and is unique \cite[Corollary 3.12]{Ste}. 


\subsection{Subsets of countable ordinals} 

In this section, we assume that $M_1^{\#}$ exists and is $\omega_1+1$-iterable. 
We show that $\CC_{\alpha}$ is equal to the power set of $\alpha$ in $M_1^{\#}$ and therefore equal to the power set of $\alpha$ in $M^{\infty}$ for countable ordinals $\alpha$. 

\begin{lemma} \label{recognizable implies in M_infty} 
If $x$ is a recognizable subset of a countable ordinal $\alpha$, 
then $x\in M^{\infty}$. 
\end{lemma} 
\begin{proof} 
It is sufficient to show that $x\in M^{\alpha}$, since $M^{\alpha}|\mu_\alpha= M_1|\mu_\alpha$ and $\alpha\leq \mu_\alpha$. 
There is a countable iteration tree on $M^{\alpha}$ with last model $N$ such that $x$ is $\mathbb{P}$-generic over $N$ \cite{Ste, NZ} (see also \cite[Lemma 2.42]{Schlicht}). 
Suppose that $P$ recognizes $x$ from $\beta$. 
Then $P$ recognizes $x$ from $\beta$ in $N[x]$. Suppose that this is forced by a condition $p$. 

Suppose that $y$ is $\mathbb{P}^N$-generic below $p$ over $N[x]$ in $V$. 
Then $y$ is recognized by $P$ from $\beta$ in $N[y]$. 
Hence halts on input $y$ with parameter $\beta$ in $V$ by the absoluteness of computations. This contradicts the uniqueness of $x$. 
\end{proof} 

We need the following notions. 
The \emph{length} or \emph{index}  $\mathrm{lh}(E)$ of an extender $E$ is its index in the extender sequence if $E\neq\emptyset$ and $0$ if $E=0$. 
An iteration tree is normal if the lengths of the extenders are weakly increasing. 

\begin{definition} 
Suppose that $T$ is a normal iteration tree (see \cite{Schindler}) with the sequence $\langle M_i\mid i<\lambda\rangle$ of models and the sequence $\langle E_i\mid i<\lambda\rangle$ of extenders. 
\begin{enumerate-(i)} 
\item 
The \emph{common-part model of $T$} is $M_T=\bigcup_{i<\lambda}M_i|\mathrm{lh}(E_i)$. 
\item 
The height of the common-part model is $\delta_T=\sup_{i<\lambda}\mathrm{lh}(E_i)$. 
\item 
$T$ is \emph{maximal} if $\delta_T$ is a Woodin cardinal in $L(M_T)$. 
\item 
$T$ is \emph{short} if $T$ is not maximal. 
\end{enumerate-(i)} 
\end{definition} 


\begin{definition} 
A premouse $M$ is \emph{$\Pi^1_2$-iterable} \cite{Ste2} if there is an iteration strategy $\sigma$ such that the following conditions hold for every countable iteration tree $T$ on $M$ following $\sigma$. 
\begin{enumerate-(i)} 
\item 
If $T$ has successor length, the next ultrapower is well-founded. 
\item 
If $T$ has limit length, then for every $\alpha<\omega_1$ there is a cofinal branch $b$ such that $\alpha$ is contained in the well-founded branch of the direct limit model $M_b$. 
\end{enumerate-(i)} 
\end{definition} 

The following 
result follows from \cite[Lemma 2.2]{Ste2}. 

\begin{lemma} \label{in M_infty implies recognizable} 
Suppose that $x\in M^{\infty}$ is a subset of a countable ordinal $\alpha$. Then $x$ is computable from a recognizable subset of a countable ordinal. 
\end{lemma} 

\begin{proof} 
Suppose that $\kappa$ is the least measurable cardinal in $M_1$ and $\mu=(\kappa^+)^{M_1}$.  
Let $M=M_1| \mu$. 
We first suppose that $x\in M$. The proof for the general case is analogous, where $M$ is replaced with an iterate of $M$ such that $\kappa$ is mapped above $\alpha$. 

Suppose that $N$ is a countable solid \cite[2.7.4]{MiSte} and sound \cite[2.8.3]{MiSte} $\Pi^1_2$-iterable premouse of height $\mu$. 
Suppose that $N$ is a model of $\mathsf{ZF}^-$ with a largest cardinal $\kappa$ and no measurable cardinals. 
In the rest of the proof we show that $M=N$. 
This implies the conclusion as follows. 
Suppose that $c$ is the subset of $\gamma$ which codes $M$ via G\"odel pairing, where $\gamma$ is the order type of the canonical well-order of $M$. Then $c$ is recognizable. 

We form the co-iteration of $M$ and $N$ to $\bar{M}$ and $\bar{N}$. We will show that $M=\bar{M}=N=\bar{N}$. 
Let $T$ denote the tree on $M$ and let $U$ denote the tree on $N$. 

The co-iteration terminates if the target models are comparable and well-founded or one of the models is ill-founded. 
If the target models are comparable, we say that the co-iteration succeeds. 
We call an iteration \emph{simple} if there is no drop along the main branch. 

%




\begin{claim*} 
If all initial segments of $T$ and $U$ are short, then the co-iteration succeeds. 
\end{claim*} 
\begin{proof} 
Since the trees are short, there is at most one well-founded branch, since otherwise $\delta_T$ is a Woodin cardinal in $M_T$ \cite[Theorem 6.10]{Ste}. 

Note that every limit initial segment of $T$ has a well-founded branch, since $M$ is $\omega_1+1$-iterable. 
Suppose that some limit initial segment $U\upharpoonright \gamma$ of $U$ has no well-founded branch. 
Since $N$ is $\Pi^1_2$-iterable, for every $\alpha<\omega_1$, there is a cofinal branch $b$ in $U\upharpoonright \gamma$ such that $\alpha\subseteq \mathrm{wfp}(N_b)$, where $N_b$ denotes the model for the branch $b$. 
Hence $\delta_{U\upharpoonright \gamma}$ is Woodin in $L_{\alpha}(M_{U\upharpoonright \gamma})$ \cite[Theorem 6.10]{Ste}. 
Since this holds for all countable $\alpha$, $\delta_{U\upharpoonright \gamma}$ is Woodin in $L(M_{U\upharpoonright \gamma})$, contradicting the assumption. 

Then the models in step $\omega_1$ are comparable and 
the pressing down argument in the proof of the comparison lemma shows that the comparison is successful at some countable step. 
\end{proof} 

\begin{claim*} 
If all initial segments of $T$ and $U$ are short, then $M=N$. 
\end{claim*} 
\begin{proof} 
First suppose that $\bar{M}\lhd \bar{N}$. 
Let $i$ denote the least truncation point on either side. 
Suppose that $\kappa_i$ is the critical point of the extender $E^{M_i}_{\nu_i}$ used in the iteration of $M$ in step $i$. 

Then $\rho_k(M_i)\leq \kappa_i$ for some $k$, since we cut down in step $i$. 
Since 
$M_i$ is $k$-sound above $\kappa_i$, there is a well-order of order type $\mathrm{Ord}^{M_i}$ computable from a set $A^k_{M_i}$ (the master code) such that $A^k_{M_i}$ is definable over $M_i$, and therefore over $\bar{M}$. 
Then $A^k_{M_i}\in \bar{N}$. 
Since $A^k_{M_i}\in \bar{N}$ codes a well-order of the height of $M_i$, and therefore a well-order with order type $\nu_i$, $\nu_i$ is collapsed in $\bar{N}$. 
This is a contradiction, since by the usual properties of the index, $\nu_i$ is a cardinal in all iterates of $N_i$ and therefore in $\bar{N}$. 

If $\bar{N}\lhd \bar{M}$ then the argument is symmetric. 

Now suppose that $\bar{M}=\bar{N}$. 
In this case, the proof works by the following standard argument. 
Suppose that $i$ is the last truncation point of the iteration of $M$ and $j$ is the last truncation point of the iteration of $N$. 
Suppose that $M'$ is the truncation of $M_i$ and $N'$ is the truncation of $N_j$. 
Since $M_i$ is sound, the core of $\bar{M}$ is equal to $M'$. 
Since $N_j$ is sound, the core of $\bar{N}$ is equal to $N'$. 
Recall that the core is the hull of the projectum and the standard parameter. Since we iterate above the projectum, the core does not change. 
Hence $M'=N'$. 

First suppose that $i<j$. 
Then $E^{M_i}_{\nu_i}\neq\emptyset$, since we truncated in step $i$. 
Since $i<j$, $E_{\nu_i}^{M_j}=E_{\nu_i}^{N_j}=\emptyset$. 
Since $M'=N'$, $E_{\nu_i}^{N_j}=E_{\nu_i}^{M_i}$. This contradicts the fact that $E_{\nu_i}^{M_i}\neq\emptyset$ and $E_{\nu_i}^{N_j}=\emptyset$. 
Second suppose that $i=j$. 
Since $E_{\nu_i}^{M_i}\neq E_{\nu_i}^{N_i}$ and since these extenders appear in $M'$ and $N'$, we have $M'\neq N'$, contradicting the assumption. 
\end{proof}

\begin{claim*} \label{the iteration always uses the least normal measure}
Suppose that $T$ is maximal. Then in $U$, there are finitely many drops, and each step is the ultrapower with the unique normal measure on the least measurable cardinal. 
Moreover the co-iteration succeeds. 
\end{claim*} 
\begin{proof} 
There is no drop along the main branch of $T$ by Claim \ref{no drop in maximal trees}. 
Therefore every step in $U$ is the ultrapower by the unique normal measure on the least measurable cardinal, and this measurable cardinal has order $0$. 
Otherwise there would be a measurable cardinal in $M$. 

Hence all iterates of $N$ are well-founded. 
Then the models in step $\omega_1$ are comparable and 
the pressing down argument in the proof of the comparison lemma shows that the comparison is successful at some countable step. 
\end{proof}

\begin{claim*} \label{no drop in maximal trees} 
If $T$ is maximal, 
then there is no drop along the main branch of $T$. 
The same holds for $U$. 
\end{claim*} 
\begin{proof} 
We will prove this for $T$. The proof for $U$ is analogous. 
Suppose that the last truncation in $T$ is at $i$. 
Suppose that $M_i$ is cut down to $\bar{M}_i$ and $\bar{M}$ is a simple iterate of $\bar{M}_i$. 
We argue that there is no Woodin cardinal in $L[\bar{M}]$. 

In $M_1$ there is no extender above a Woodin cardinal. 
Therefore $M_1$ is tame, i.e. there is no extender overlapping a Woodin cardinal. Otherwise in the ultrapower, the image of the critical point is measurable and the Woodin cardinal remains a Woodin cardinal, since the extender sequence is coherent. 

Hence the extender $E$ in $M$ responsible for the last drop does not overlap a Woodin cardinal. 

\begin{subclaim*} 
In $\bar{M}_i$ there is no cardinal $\delta$ which is a Woodin cardinal definably over $\bar{M}_i$. 
\end{subclaim*} 
\begin{proof} 
Suppose that $\delta$ is below $\mathrm{crit}(E)$. Then $M$ and $M_1$ have a measure above a Woodin cardinal, contradicting the definition of $M_1$. 

Suppose that $\delta$ is above $\mathrm{crit}(E)$. 
Since $M_i$ is acceptable, there is a surjection from $\mathrm{crit}(E)$ onto $\bar{M}_i$ definable over $\bar{M}_i$. Then the Woodin cardinal is destroyed definably over $\bar{M}_i$. The same holds for $\bar{M}$ by elementarily. 
\end{proof} 
This proves the claim. 
\end{proof}

\begin{claim*} 
Suppose that $T$ or $U$ is maximal. 
Then $M=N$. 
\end{claim*} 
\begin{proof} 
We prove this for $T$. The proof for $U$ is analogous. 
Suppose that the iteration tree $U$ is nontrivial. Since there are no measurable cardinals in $N$, this implies that there is a drop in $U$. 
Hence there is a bounded subset $x$ of $\kappa$ which is not an element of $N$, but it is definable over $N$. 

The iteration of $N$ is definable in $N$ by Claim \ref{the iteration always uses the least normal measure}. 
Since $N$ is a model of $\mathsf{ZF}^-$, the iterate $N_\alpha$ of $N$ in any step $\leq\kappa$ 
is an element of $N$ and hence it has height $<\mu$, where $\mu=(\kappa^+)^{M_1}$. 
Hence $N_\alpha$ is not an initial segment of $M$. 

Let $N_\kappa$ denote the iterate of $N$ in step $\kappa$. 
Since $x$ is definable over $N_\kappa$ but it is not an element of $N_\kappa$, $N_\kappa$ is not an initial segment of $M$. 
Therefore the co-iteration continues.  
The index of the next extender used in $N_\kappa$ is a cardinal in the following iterates of $N_\kappa$. 
However there are no cardinals above $\kappa$ in $M$ and therefore the final models cannot agree. 
This contradicts the fact that the co-iteration terminates by Claim \ref{the iteration always uses the least normal measure}. 
\end{proof} 

 
This completes the proof. 
\end{proof} 

\begin{thm} 
If $\alpha$ is countable and $x\subseteq \alpha$, then $x\in \CC_{\alpha}$ if and only if $x\in M^{\infty}$. 
In particular $\CC_{\omega}=P(\omega)\cap M_1=Q_3$ \cite{KeMaSo}. 
\end{thm} 
\begin{proof} 
This follows from Lemma \ref{recognizable implies in M_infty} and Lemma \ref{in M_infty implies recognizable}.
\end{proof} 

\begin{remark} 
Suppose that $\gamma$ is countable. Then the recognizable jump $J^{\gamma}$ with parameter $\gamma$ is in $M_1$. 
\end{remark} 
\begin{proof} 
Suppose that $N$ is a countable iterate of $M_1$ such that its least measure is above $\gamma$. 
Suppose that $P$ does not recognize any real. It follows from the genericity iteration that this holds if and only if it is forced over $N$ by the extender algebra. Since we do not move $\gamma$ in iterations to make reals generic over iterates of $N$, this is a statement in $N$ in the parameter $\gamma$. 
Therefore we can determine in $N$ whether a given program $P$ recognizes a real. So the recognizable jump is in $N$. Since $M_1$ and $N$ have the same reals, the recognizable jump is in $M_1$. 
\end{proof} 


%
%

\subsection{Subsets of $\omega_1$}

In this section, we show that it is consistent with 
the existence of an $\omega_1$-iterable $M_n^{\#}(x)$ for all reals $x$ and all $n$ and therefore with projective determinacy 
that every recognizable subset of $\omega_1$ is in $M_1$. 

In this section we assume that an $M_1^{\#}(x)$ exists for all reals $x$. Then $M_1^{\#}$ is unique by \cite[Lemma 1.2.26]{Schlicht} or \cite[Corollary 1.8]{Nee}.

\begin{thm} 
Suppose that $\mathsf{ZF}+\mathsf{DC}$ holds. 
Suppose that for every $A\subseteq\omega_1$ there is some real $x$ with $A\in L[x]$. 
Suppose that $\mathbb{P}$ is homogeneous and preserves $\omega_1$ and that $G$ is $\mathbb{P}$-generic over $V$. 
Then in $V[G]$, every recognizable subset of $\omega_1$ is in $M^{\infty}$. 
\end{thm} 
\begin{proof} 
Suppose that $A\in V[G]$ is a recognizable subset of $\omega_1$. 
Then $A\in V$ by Lemma \ref{weakly homogeneous}. 
We work in $V$. 
Suppose that $x$ is a real with $A\in L[x]$. 

We first suppose that there is some $\beta<\omega_1$ and a countable iteration with embedding $\pi\colon M^{\beta}\rightarrow N$ such that $x$ is generic over $N$ and $A\notin N$. 
Suppose that $\dot{A}$ is a name in $N$ for $A$. 
Then there is a condition $p$ in $N$ which forces that $P$ recognizes $\dot{A}$ from $\alpha$ and that $A\notin N$. 
Suppose that $y$ is a generic filter in $V$ over $N[x]$ in below $p$. 
Let $B=\dot{A}^y$. 
Since $x,y$ are mutually generic, $N[x]\cap N[y]= N$. 
Then $A\neq B$. 
But $P$ recognizes $A,B$ from $\alpha$, contradicting the uniqueness of $A$. 

Now suppose that there is no such iteration. 
\begin{claim*} 
$A\cap \beta \in M^{\infty}$ for all $\beta<\omega_1$. 
\end{claim*} 
\begin{proof} 
Suppose that $\beta\leq\gamma<\omega_1$. 
There is a countable iteration above $\gamma$ with embedding $\pi\colon M^{\gamma}\rightarrow N$ such that $x$ is generic over $N$. 
Then $A\in N$. 
Since $P(\beta)\cap M^{\infty}= P(\beta)\cap M^{\gamma}= P(\beta)\cap N$, this implies that $A\cap \beta\in M^{\infty}$. 
\end{proof} 

Suppose that $\kappa$ is the least measurable cardinal in $M_1$. 
Suppose that $\langle \kappa_{\alpha}\mid \alpha<\omega_1 \rangle$ is the sequence of images of $\kappa$. 
Let $\pi_{\alpha,\beta}\colon M^{\alpha}\rightarrow M^{\beta}$ denote the elementary embeddings. 
For every limit $\lambda<\omega_1$, there is some $f(\lambda)<\lambda$ such that $A\cap \kappa_\lambda$ has a preimage in $M^{f(\lambda)}$. 
Then there is a stationary set $S\subseteq\omega_1$ and some $\alpha<\omega_1$ such that $A\cap \kappa_\beta$ has the same preimage $\bar{A}$ in $M^{\alpha}$ for all $\beta\in S$. 
Then $\pi_{\alpha,\beta}(\bar{A})=A\cap \beta$ for all $\alpha\leq\beta <\omega_1$. 
Since $\pi_{\alpha,\beta}\in M^{\alpha}$ for all $\alpha\leq\beta<\omega_1$, this implies that $A\in M^{\alpha}$. 
Therefore $A\in M^{\infty}$. 
%
\end{proof} 

The assumption in the previous theorem hold if the ground model $L(\mathbb{R})$ satisfies the axiom of determinacy \cite[Theorem 28.5]{Kan} 
and for the $\mathbb{P}_{\mathrm{max}}$-extension of $L(\mathbb{R})$. 

The following result shows that a Woodin cardinal is not sufficient to prove the statement: 
if $M$ is 
an inner model with a Woodin cardinal which is iterable for countable short trees, then 
every recognizable subset of $\omega_1$ is in $M$. 

\begin{lemma}  
Suppose that an $\omega_1+1$-iterable $M_1^{\#}$ exists. Then there is a cardinal-preserving generic extension $N$ of $M_1$ such that in $N$, there is a recognizable set of ordinals which is not in $M_1$. 
\end{lemma} 
\begin{proof} 
We work in $M_1$. 
Suppose that $T$ is a Suslin tree with the unique branch property. 
The existence of such a tree is proved from $\diamond$ in \cite[Theorem 1.1]{FuHa}.
Suppose that $b$ is a $T$-generic branch over $M_1$. 
Let $N=M_1[b]$. 
Since $T$ is $<\omega_1$-distributive, 
the forcing does not change cofinalities. 
There is a unique branch in $T$ of length $\omega_1$ in $N$. 
Suppose that $\kappa$ is the least measurable cardinal in $M_1$ and $x$ is the subset of $\kappa^+$ which codes $M_1|\kappa^+$ with its canonical well-order via G\"odel pairing. 
Then $x$ is recognizable by the proof of Lemma \ref{in M_infty implies recognizable}. 
Therefore the join $x\oplus b$ is recognizable but $x\oplus b$ is not an element of $M_1$. 
\end{proof} 


\subsection{Structures which are not recognizable} 

The main open question is whether the results in the previous section hold for arbitrary sets of ordinals. 

\begin{question} 
Is it consistent that an $\omega_1+1$-iterable $M_1^{\#}$ exists and that there is a recognizable subset of $\omega_1$ which is not in $M^{\infty}$? 
\end{question} 

In this section we will show that certain sets of ordinals cannot be recognizable. 
The sets code transitive models with infinitely many measurable cardinals. 

\begin{lemma}\label{fixed points of an ultrapower} 
Suppose that the $\mathrm{GCH}$ holds. 
Suppose that $\kappa$ is measurable, $U$ is a $<\kappa$-complete ultrafilter on $\kappa$ and $j\colon V\rightarrow M$ is the ultrapower map for $U$. 
Suppose that $\lambda$ is an infinite cardinal. 
Then $j(\lambda)>\lambda$ if and only if $\cof(\lambda)=\kappa$ or $\lambda=\mu^+$ and $\cof(\mu)=\kappa$. 
\end{lemma} 
\begin{proof} 
We have $\kappa^+<j(\kappa)<\kappa^{++}$. 
Suppose that $\lambda>\kappa$, $\cof(\lambda)=\kappa$ and $\langle \lambda_\alpha \mid \alpha<\kappa\rangle$ is a cofinal strictly increasing sequence of regular cardinals below $\lambda$. 
Then the linear order $(\lambda^+,<)$ embeds into $(\prod_{\alpha<\kappa} \lambda_\alpha,<_U)$, since at cofinalities below $\lambda_\alpha$ we can form the pointwise supremum in all coordinates above $\alpha$ and choose the coordinates below $\alpha$ arbitrarily. 
Hence $(\lambda^+,<)$ embeds into $({}^{\kappa}\lambda,<_U)$ and $j(\lambda)>\lambda^+$. 

If $\lambda=\mu^+$ and $\cof(\mu)=\kappa$, then $j(\lambda)>j(\mu)\geq \lambda$. 
We prove the remaining cases by induction. 
First suppose that $\lambda$ is a limit cardinal with $\cof(\lambda)<\kappa$. 
Then $j(\lambda)=\sup_{\alpha<\lambda}j(\alpha)=\lambda$. 
Second, suppose that $\lambda$ is a limit cardinal with $\cof(\lambda)>\kappa$. Suppose that $f\colon \kappa\rightarrow \lambda$. Then there is some $\alpha<\lambda$ such that $f\colon \kappa\rightarrow \alpha$. Hence $[f]\leq j(\alpha)$ and $j(\lambda)=\sup_{\alpha<\lambda}j(\alpha)$. 

Third, suppose that $\lambda=\mu^+$ and $\mu$ is a limit cardinal with $\mu>\kappa$ and $\cof(\mu)<\kappa$. 
If $f\colon \kappa\rightarrow \lambda$, then there is some $\alpha<\lambda$ such that $f=_U c_\alpha$, where $c_\alpha\colon\kappa\rightarrow \lambda$ is the constant function with value $\alpha$. 
Suppose that $\alpha=\bigcup_{\beta<\cof(\mu)}Y_\alpha$ where $|Y_\beta|<\mu$. 
Suppose that $g<_U c_\alpha$. 
Then there is some $X\in U$ and some $\beta<\cof(\mu)$ such that $g[X]\subseteq Y_\beta$. 
Since $\nu^{\kappa}<\mu$ for all $\nu<\mu$, this implies that $j(\alpha)=[c_\alpha]$ has size at most $\mu$. 
Hence $j(\lambda)=\sup_{\alpha<\lambda} j(\alpha)=\lambda$. 

In the last case, suppose that $\lambda=\mu^+$ and $\mu$ is a limit cardinal with $\mu>\kappa$ and $\cof(\mu)> \kappa$, then $\mu^{\kappa}=\mu$. Hence $j(\alpha)<\lambda$ for all $\alpha<\lambda$. 
Therefore $j(\lambda)=\lambda$. 
Note that the last case also covers the case $\lambda=\mu^{++}$, where $\mu$ is a limit cardinal with $\mu>\kappa$ and $\cof(\mu)=\kappa$. 
\end{proof} 

\begin{lemma}\label{bound on iterated ultrapower of kappa}  
Suppose that the $\mathrm{GCH}$ holds. 
Suppose that $\kappa$ is measurable, $U$ is a $<\kappa$-complete ultrafilter on $\kappa$, and $j\colon V\rightarrow M$ is the ultrapower map for $U$. 
Then $\kappa^+<j^i(\kappa)<\kappa^{++}$ and $\cof(j^i(\kappa))=\kappa^+$ for all $i\in\omega$. 
\end{lemma} 
\begin{proof} 
We have $\kappa^+<j(\kappa)<\kappa^{++}$. For the first claim, it is sufficient to show that $j(\alpha)<\kappa^{++}$ for all $\alpha<\kappa^{++}$. 
Suppose that $\alpha<\kappa^{++}$. 
Then there are only $(\kappa^+)^{\kappa}=\kappa^+$ many functions $f\colon \kappa\rightarrow \alpha$ with $[f]<_U[c_\alpha]$ and hence $j(\alpha)<\kappa^{++}$. 

For the second claim, $\cof(j(\kappa))$ cannot be $\kappa$, since we can diagonalize against a sequence $\langle f_\alpha\mid \alpha<\kappa\rangle$ of functions $f_\alpha\colon \kappa\rightarrow \kappa$ to construct an upper $<_U$-bound. 
Hence $\cof(j(\kappa))=\kappa^+$. 
Suppose that $\cof(j^i(\kappa))=\kappa^+$. Then every sequence $\langle f_\alpha\mid \alpha<\kappa\rangle$ of functions $f_\alpha\colon \kappa\rightarrow j^i(\kappa)$ has a pointwise upper bound and therefore a $<_U$-upper bound. 
Hence $\cof(j^{i+1}(\kappa))=\kappa^+$. 
\end{proof} 

\begin{lemma}\label{fixed points of an ultrapower} 
Suppose that the $\mathrm{GCH}$ holds. 
Suppose that $\kappa$ is measurable, $U$ is a $<\kappa$-complete ultrafilter on $\kappa$, and $j\colon V\rightarrow M$ is the ultrapower map for $U$. 
Suppose that $\lambda$ is an ordinal with $j[\lambda]\subseteq \lambda$. 
Then $j(\lambda)>\lambda$ if and only if $\cof(\lambda)=\kappa$. 
\end{lemma} 
\begin{proof} 
Suppose that $\lambda$ is a limit ordinal with $\cof(\lambda)<\kappa$. 
Since $j[\lambda]\subseteq\lambda$, $j(\lambda)=\sup_{\alpha<\lambda}j(\alpha)=\lambda$. 
Firstly, suppose that $\lambda$ is a limit ordinal with $\cof(\lambda)>\kappa$. Suppose that $f\colon \kappa\rightarrow \lambda$. Then there is some $\alpha<\lambda$ such that $f\colon \kappa\rightarrow \alpha$. Hence $[f]\leq j(\alpha)$. 
There are unboundedly many $\alpha<\lambda$ with $j[\alpha]\subseteq \alpha$ and $\cof(\alpha)=\omega$ and hence $j(\alpha)=\alpha$ by the previous case. 
Hence $j(\lambda)=\sup_{\alpha<\lambda}j(\alpha)=\lambda$. 

Secondly, suppose that $\lambda$ is a limit ordinal with $\cof(\lambda)=\kappa$. 
As in the previous case, there are unboundedly many $\alpha<\kappa$ with $j(\alpha)=\alpha$. 
Suppose that $\langle \lambda_\alpha\mid \alpha<\kappa\rangle$ is a strictly increasing cofinal sequence in $\lambda$ of fixed points of $j$. 
The image of this sequence is strictly longer than $\kappa$ and hence $j(\lambda)>\lambda$. 
\end{proof}

\begin{lemma}\label{fixed points of one ultrapower from infinitely many} 
Suppose that the $\mathrm{GCH}$ holds. 
Suppose that $\langle \kappa_n\mid n\in\omega\rangle$ is a strictly increasing sequence of measurable cardinals. Suppose that $U_n$ is a $<\kappa_n$-complete ultrafilter on $\kappa_n$ and $j_n\colon V\rightarrow M_n$ is the ultrapower with $U_n$. 
For every ordinal $\alpha$, there is some $n$ with $j(\alpha)=\alpha$. 
\end{lemma} 
\begin{proof} 
Suppose that $\alpha$ is the least ordinal such that $j_n(\alpha)>\alpha$ for all $n\in\omega$. 
Suppose that $\beta_n\leq \alpha$ is least such that $j_n^k(\beta_n)>\alpha$ for some $k$. 
Further suppose that $k_n$ is minimal with this property. 

\begin{claim}\label{beta is closed under the embedding}  
$j_n[\beta_n]\subseteq \beta_n$ for all $n$. 
\end{claim} 
\begin{proof} 
Suppose that $j_n(\beta)\geq \beta_n$ for some $\beta<\beta_n$. 
Then $j_n^{k_n+1}(\beta)>\alpha$, contradicting the minimality of $\beta_n$. 
\end{proof} 

Claim \ref{beta is closed under the embedding} and Lemma \ref{fixed points of an ultrapower} imply that $\cof(\beta_n)=\kappa_n$ for each $n$. 
Let $\gamma_n=j_n^{k_n}(\beta_n)$. 
Then $\cof(\gamma_n)=\kappa_n^+$ by Lemma \ref{bound on iterated ultrapower of kappa}. 
Therefore there is a strictly increasing sequence $\langle n_i\mid i\in\omega\rangle$ such that 
the sequences $\langle \beta_{n_i}\mid i\in\omega\rangle$ and $\langle \gamma_{n_i}\mid i\in\omega\rangle$ are strictly increasing. 

\begin{claim}\label{configuration of images of measurables} 
Suppose that $\mu$, $\nu$ are measurable with $\mu<\nu$, $U_\mu$, $U_\nu$ are $<\mu$-complete and $\nu$-complete ultrafilters on $\mu$, $\nu$ and $j_\mu$, $j_\nu$ are the ultrapower maps for $U_\mu$, $U_\nu$. 
Suppose that $j_\mu[\beta_\mu]\subseteq \beta_\mu$ and $j_{\nu}[\beta_\nu]\subseteq \beta_\nu$. Then the configuration $\beta_\mu<\beta_\nu<j_\mu^k(\beta_\mu)<j_\nu^l(\beta_\nu)$ is impossible for all $k,l\in\omega$. 
\end{claim} 
\begin{proof} 
Since $j_\mu(\beta_\mu)>\beta_\mu$ and $j_\nu(\beta_\nu)>\beta_\nu$, $\cof(\beta_\mu)=\mu$ and $\cof(\beta_\nu)=\nu$ by Lemma \ref{fixed points of an ultrapower}. 
Hence $j_\nu(\beta_\mu)=\beta_\mu$. 
Let $N=Ult^l(V, U_\nu)$. 
Then $N^{\nu}\subseteq N$. 
and hence $U_\mu\in N$ and $(j_\mu^k(\beta_\mu))^N=j_\mu^k(\beta_\mu)$ for all $k\in\omega$. 
Moreover $j_\nu^l$ does not move $j_\mu^k(\beta_\mu)$. 
Since $\beta_\nu<j_\mu^k(\beta_\mu)$, by elementarity $j_\nu^l(\beta_\nu)<j_\mu^l(\beta_\mu)$, contradicting the assumption. 
\end{proof} 
Let $\mu=\kappa_{n_0}$, $\nu=\kappa_{n_1}$, $\beta_\mu=\beta_{n_0}$, and $\beta_\nu=\beta_{n_1}$ in Claim \ref{configuration of images of measurables}. This contradicts the assumptions on $\beta_{n_0}$, $\beta_{n_1}$. 
\end{proof} 

\begin{lemma} \label{non-recognizable structure} 
Suppose that $N$ is a transitive model of $\mathsf{ZFC}$ with $\mathrm{Ord}\subseteq N$. Suppose that $\langle \kappa_n\mid n\in\omega\rangle$ is a strictly increasing sequence of measurable cardinals in $N$ with $\sup_n \kappa_n=\kappa$. 
Suppose that $x$ is a set of ordinals in $N$ with $V_\kappa^N\in L[x]$. 
Then $x$ is not recognizable. 
\end{lemma} 
\begin{proof} 
Suppose that $x$ is recognized by $P$ from $\alpha$. 
Since $\mathrm{Ord}\subseteq N$, $x$ is recognized by $P$ from $\alpha$ in $N$. 
There is a $<\kappa_n$-complete ultrafilter $U$ on $\kappa_n$ in $N$ such that $j_n(\alpha)=\alpha$ by Lemma \ref{fixed points of one ultrapower from infinitely many}, where $j_n\colon N\rightarrow I$ is the ultrapower embedding. 
Then $P$ accepts $j(x)$ from $\alpha$ in $N$ and therefore in $V$. 
We have $j(x)\neq x$, since $V_\kappa^N\in L[x]$ and therefore $U\in L[x]$, but $U\notin I$. 
This contradicts the uniqueness of $x$. 
\end{proof} 

%
%

\subsection{Generically recognizable sets} 

The following version of recognizability defines exactly the sets of ordinals in $M^{\infty}$. 

\begin{definition} 
\begin{enumerate-(1)} 
\item 
A subset $x$ of an ordinal $\alpha$ is \emph{generically recognizable} if it is recognizable in all $Col(\omega,\beta)$-generic extensions for sufficiently large ordinals $\beta$. 
\item 
The \emph{generically recognizable closure} is the class of all sets $x$ of ordinals such that $x$ is computable from a generically recognizable set of ordinals. 
\end{enumerate-(1)} 
\end{definition} 

\begin{lemma} 
Suppose that there is a proper class of Woodin cardinals. 
Then the elements of the generically recognizable closure are exactly the sets of ordinals in $M^{\infty}$. 
\end{lemma} 
\begin{proof} 
Since there is a proper class of Woodin cardinals, the model 
 $M^{\infty}$ is generically absolute. 
 
Suppose that $x$ is a generically recognizable subset of an ordinal $\gamma$ and that $G$ is $Col(\omega,\gamma)$-generic over $V$. 
Then $x$ is in $M^{\infty}$ in $V[G]$ and therefore in $V$. 

Suppose that $x$ is a subset of $\alpha$ in $M^{\infty}$. 
Suppose that $\kappa\geq\alpha$ is an image of the critical point in the iteration defining $M^{\infty}$. 
Suppose that $M$ is an initial segment of $M^{\infty}$ of height $\gamma=(\kappa^+)^{M^{\infty}}$ and that $G$ is $Col(\omega,\gamma)$-generic over $V$. 
Then in $V[G]$ the code for $M$ given by its canonical well-order is recognizable by the proof of Lemma \ref{in M_infty implies recognizable}.
Hence $x$ is an element of the generically recognizable closure. 
\end{proof} 


\section{Conclusion and questions} 

We have obtained connections between recognizable sets and the inner model $M_1$ and in particular obtained computational characterizations of initial segments of $M_1$ far beyond $L$. Moreover we have seen that the recognizable closure is related to implicit definability. 
%

We conclude with the main open questions. 

\begin{question} 
Suppose that there is a proper class of Woodin cardinals. 
Is every recognizable set of ordinals in $M_1$? 
\end{question} 

The following is an approach to refute this conjecture. 

\begin{question} 
Is it possible to start from a model with large cardinals and by class forcing obtain that $V$, $HOD$ and the recognizable closure for subsets of arbitrary ordinals are equal? 
\end{question}

If $M_1^{\#}$ exists, then the set of recognizable reals is countable. 

\begin{question} 
Is it consistent that there are $\omega_2$ many recognizable reals? 
\end{question} 



\begin{question} 
Suppose that $M_1^{\#}$ exists. 
Is every set of ordinals in $M_1$ constructible from a recognizable set? 
\end{question}

We ask whether similar results hold for stronger models of computation. 

\begin{question}
Suppose that an $\omega_1+1$-iterable $M_2^{\#}$ exists. If we add the operator mapping $x$ to $M_1^\#(x)$ as a possible step in computations, then is a subset of $\omega$ 
recognizable by such a machine if and only if it is in $M_2$? 
\end{question}

Another way to strengthen the model of computation is to allow countable sets of ordinals as parameters. 
This motivates the following question. 

\begin{question} 
Suppose that countable sets of ordinals are allowed as parameters instead of ordinals. Does the recognizable closure contain sets beyond the Chang model $L(\mathrm{Ord}^{\omega})$? 
\end{question} 

We define a set to be in the \emph{recognizable hull} $\RR$ if it is an element of the transitive collapse of a well-founded extensional relation 
on an ordinal such that its image under G\"odel pairing is recognizable. 
 
\begin{question} 
Which axioms of set theory hold in $\RR$, and in particular does $\RR$ always satisfy $\Sigma_1$-collection? 
\end{question}




%



\nocite{*} 

\bibliographystyle{alpha}
 \bibliography{recognizable}

\end{document}